\documentclass[a4paper,11pt]{article}

\usepackage{amsmath,amsthm,amssymb,amsfonts}
\usepackage[margin=1in]{geometry} 
\usepackage[english]{babel}
\usepackage[utf8]{inputenc}
\usepackage{float}
\usepackage{tensor}
\usepackage{csquotes}% Recommended
\usepackage{enumitem}
\usepackage{mathtools}
\usepackage{setspace}
\usepackage{mathrsfs}
\usepackage{booktabs}
\usepackage{bbm}
\usepackage{graphicx}
\usepackage{caption}
\usepackage{subcaption}
\usepackage{array}
\usepackage[
  backend=bibtex,
  style=authoryear,
  maxcitenames=2,   % replaces maxnames=2: applies only to citations
  mincitenames=1,   % "et al." keeps 1 name when truncating
  maxbibnames=999,  % bibliography always shows all authors
  uniquename=false,
  natbib=true,
  doi = false,
  url = false,
  isbn = false
]{biblatex}

\newtheorem{theorem}{Theorem}
\newtheorem{proposition}[theorem]{Proposition}
\newtheorem{lemma}[theorem]{Lemma}
\newtheorem{corollary}[theorem]{Corollary}
\newtheorem{definition}[theorem]{Definition}

\newtheorem{example}[theorem]{Example}

\AfterEndEnvironment{theorem}{\noindent\ignorespaces}
\AfterEndEnvironment{corollary}{\noindent\ignorespaces}
\AfterEndEnvironment{lemma}{\noindent\ignorespaces}
\AfterEndEnvironment{definition}{\noindent\ignorespaces}
\AfterEndEnvironment{proposition}{\noindent\ignorespaces}
\AfterEndEnvironment{assumption}{\noindent\ignorespaces}
\AfterEndEnvironment{example}{\noindent\ignorespaces}

\newcommand{\ebb}{\mathbb{E}}
\newcommand{\pbb}{\mathbb{P}}
\newcommand{\nbb}{\mathbb{N}}
\newcommand{\rbb}{\mathbb{R}}

\newcommand{\hbb}{\mathbb{H}}
\newcommand{\Fcal}{\mathcal{F}}

\newcommand{\Pcal}{\mathcal{P}}

\newcommand{\Bcal}{\mathcal{B}}

\newcommand{\Rcal}{\mathcal{R}}
\newcommand{\Scal}{\mathcal{S}}

\newcommand{\Lcal}{\mathcal{L}}

\newcommand{\Acal}{\mathcal{A}}
\newcommand{\Mcal}{\mathcal{M}}

\newcommand{\Ycal}{\mathcal{Y}}

\DeclareMathOperator{\arcosh}{arcosh}

\newcommand\norm[1]{\left\lVert#1\right\rVert}
\newcommand\abs[1]{\left\lvert#1\right\rvert}
\newcommand{\indep}{\perp \!\!\! \perp}

\title{Boundedly exchangeable graphs}
\author{
  Leon-Roman Rinke \thanks{Corresponding author.
    Leibniz Universit\"at Hannover, House of Insurance, Welfengarten 1,
    30167 Hannover, Germany. E-mail:
    {\tt leon-roman.rinke@insurance.uni-hannover.de}}
  \hspace{2cm} Stefan Weber \\[1.0ex]
  \textit{Leibniz Universit\"at Hannover}
}
\date{\today}

\begin{document}

\citetrackerfalse

\maketitle

\begin{abstract}
An infinite jointly exchangeable adjacency matrix is a mixture of graphon models, and its graph is almost surely empty or dense; classical matrix exchangeability therefore excludes sparse nonempty graphs. We study a point-process alternative, \textit{graph processes}, whose vertices are identified with the atoms of a locally finite point process on a Polish latent space and whose edges are drawn conditionally independently from a symmetric connection kernel $W$. For a Poisson graph process with diffuse $\sigma$-finite intensity, we show that its edge measure is jointly exchangeable as a random measure if and only if $W$ is essentially constant. Non-constant kernels therefore require a different exchangeability notion. We introduce \textit{bounded exchangeability}, under which every nontrivial bounded restriction determines an infinite jointly exchangeable adjacency matrix that is unique in distribution, with an associated local graphon representation. We then study \textit{threshold models} on homogeneous Riemannian manifolds, in which Poisson-distributed vertices are connected whenever their intrinsic distance is at most a fixed threshold. On every infinite-volume manifold in this class, the model is sparse, its empirical degree distribution converges to a Poisson law, and its limiting clustering coefficient depends on an additional geometric parameter that can vary while the limiting degree law is held fixed.
\end{abstract}

\vspace{0.2cm}
\noindent\textbf{Keywords:} Random graphs, infinite graphs, exchangeable arrays, latent space models, random connection models

\vspace{0.2cm}
\noindent\textbf{2020 Mathematics Subject Classification:} Primary 05C80, 60G09; secondary 60D05, 60G55, 05C63.

\newpage
\newpage

\section{Introduction}

Random graph models often impose exchangeability to express the idea that vertex labels carry no information. For an infinite adjacency matrix $A=(A_{ij})_{i,j\geq 1}$, joint exchangeability means invariance in distribution under simultaneous permutations of rows and columns.

The Aldous--Hoover--Kallenberg representation, due to \citet{Hoover1979, Aldous1981, Kallenberg1989}, characterizes this symmetry as a mixture of graphon models: in the ergodic case, an exchangeable graph can be generated by assigning independent latent variables $U_i$ to the vertices and then connecting vertices $i$ and $j$ independently with probability $W(U_i,U_j)$. The symmetric measurable function $W$ is the corresponding graphon.

This classical symmetry has a structural limitation. Infinite jointly exchangeable graphs are almost surely either empty or dense (see, e.g., \citet{OrbanzRoy2015}), and hence do not describe sparse graphs whose edge count grows linearly in the number of vertices --- the regime of many real-world networks \citep{Newman2018}. The graphex framework of \citet{VeitchRoy2015, CaronFox2017} resolves this by changing the object of exchangeability: it imposes exchangeability not on the adjacency matrix, but on a random measure that encodes the edges. This yields sparse infinite graphs, but the resulting symmetry is no longer ordinary matrix exchangeability.

The present paper takes an alternative approach by employing a point-process construction of spatial random graphs. It demonstrates that classical matrix exchangeability is recovered on bounded observation windows. A \textit{graph process} is generated from two ingredients: a locally finite point process on a latent Polish space $S$ and a symmetric connection kernel $W:S^2\to[0,1]$. Vertices are encoded by the atoms of the point process, and, given their locations $(S_i)$, edges between distinct vertices are drawn conditionally independently as Bernoulli variables with parameters $W(S_i, S_j)$. The global graph may be sparse and non-exchangeable, but its bounded restrictions can still determine ordinary jointly exchangeable adjacency laws.

The local viewpoint is also natural in applications. Vertices are typically observed through a finite spatial, temporal, or relational window, while the underlying population is idealized as large or unbounded. Sparsity is then a global scaling feature, whereas exchangeability is often plausible only after restricting attention to a bounded frame of observation in which labels carry no information. At the same time, the latent space may carry geometric content: distance can encode similarity, exposure, or interaction cost, and the volume growth of the space affects the number of potential neighbours. This makes it natural to ask for models in which local exchangeability, sparsity, and geometric structure are treated within the same framework.

This local viewpoint separates two questions. The first is structural: when does a point-process graph admit a local form of classical matrix exchangeability? The answer developed below is bounded exchangeability, under which every nontrivial bounded restriction determines an associated infinite jointly exchangeable adjacency matrix, unique in distribution, and hence a local graphon representation. The second question is geometric: when the latent space carries intrinsic notions of volume and distance, how do these geometric quantities determine the limiting graph statistics?

We address the second question on Riemannian manifolds. The vertex process is defined from Riemannian volume, and the connection rule is defined from intrinsic distance. In the threshold model, the manifold analogue of the random plane network of \citet{Gilbert1961}, two vertices are connected exactly when their Riemannian distance is at most a fixed radius $t$. On every homogeneous infinite-volume Riemannian manifold this gives a sparse, geometrically interpretable class of graph processes whose limiting degree law is controlled by the volume $V_t$ of a ball of radius $t$, while the limiting clustering coefficient carries additional geometric information through $p_t$, the probability that two independent uniform points of such a ball lie within distance $t$. The limiting degree law depends on $\delta V_t$, whereas the limiting clustering coefficient also depends on $p_t$. Consequently, after matching $\delta V_t$ and hence the limiting degree distribution, hyperbolic space has a strictly smaller limiting clustering coefficient than Euclidean space.

\subsection{Contributions}

\begin{enumerate}[label=(\roman*)]
    \item \textbf{Graph processes and graphex exchangeability.} We formulate graph processes as point-process-driven random graphs: the vertices are encoded by the atoms of a locally finite point process on a latent Polish space, and edges between distinct vertices are drawn conditionally independently from a symmetric connection kernel $W$. Under the standing assumptions of the graphex framework, we prove that for a Poisson graph process with diffuse $\sigma$-finite intensity measure, its edge measure --- which records only present edges --- is jointly exchangeable in the random-measure sense if and only if its connection kernel is essentially constant. Hence nonconstant connection kernels are incompatible with graphex-type exchangeability within this class.
    \item \textbf{Bounded exchangeability and local graphon representations.} \sloppy We introduce \textit{bounded exchangeability} as a local form of classical matrix exchangeability. For every nontrivial bounded observation window $B$, a boundedly exchangeable graph process determines an infinite jointly exchangeable adjacency matrix $A^B$, unique in distribution. This associated exchangeable law has an Aldous--Hoover--Kallenberg representation with local graphon $W_B = \widetilde{W} \circ(f_B\otimes f_B)$, where $f_B$ samples from the local vertex distribution on $B$ and $\widetilde{W}(x,y) \coloneqq \mathbf{1}\{x \neq y\} W(x,y)$ is the zero-diagonal version of $W$. We also describe how the local graphons and exchangeable matrices transform when observation windows are enlarged, restricted, or partitioned into disjoint subwindows.
    \item \textbf{Intrinsic graph processes on Riemannian manifolds.} We single out a geometrically natural class of graph processes on complete connected $d$-dimensional Riemannian manifolds $(M, g)$, taking the vertex intensity to be $\lambda_g$ for the Riemannian volume measure $\lambda_g$ and using the intrinsic distance $d_g$ in the connection rule. For distance-based kernels $W(x,y)=h(d_g(x,y))$, we prove a scale-invariance theorem.
    \item \textbf{Sparsity, degree distribution, clustering, and curvature.} For the threshold model on homogeneous infinite-volume manifolds, we prove asymptotic sparsity and identify the limiting degree and clustering statistics. The empirical degree distribution converges pointwise to a Poisson law determined by the expected number of points in a ball of connection radius. The clustering coefficient converges in probability to a limit that depends both on this expected local number of points and on the geometry of such balls. More precisely, the remaining geometric contribution is governed by how likely two independently chosen points in a connection ball are themselves close enough to be connected. Thus the limiting degree law is fixed by the expected number of points in a connection ball, whereas the clustering limit also reflects the shape of the underlying space. In the Euclidean–hyperbolic comparison, the degree law can be matched by matching these expected local point counts, but the clustering limit remains curvature-dependent; in particular, hyperbolic space yields strictly lower limiting clustering than Euclidean space.
\end{enumerate}

\subsection{Literature}

The probabilistic study of random graphs began in earnest with the foundational models of \citet{ErdosRenyi1959} and \citet{Gilbert1959}, in which edges are formed independently between a fixed number of vertices. A spatial counterpart was introduced shortly thereafter by \citet{Gilbert1961}, whose random plane network places vertices according to a Poisson point process in the Euclidean plane and connects pairs of points lying within a fixed distance. This model initiated a large body of work on random geometric graphs and continuum percolation; see, for example, \citet{MeesterRoy1996, Penrose2003} for comprehensive accounts. For the random geometric graph, limit theorems for length-power functionals, including the edge count, were obtained by \citet{Reitzner2017}, and \citet{Grygierek2020} derived a quantitative central limit theorem for edge counts in the high-dimensional regime. More general distance-dependent connection rules are treated in the random connection model, where pairs of Poisson points are connected independently with probability depending on their separation \citep{MeesterRoy1996, MeesterPenrose1997}; central limit theorems for its counts of finite components were established by \citet{Last2021}. Limit theory for related Poisson-based graphs extends beyond the Euclidean setting: \citet{Sambale2025} prove quantitative central limit theorems for the nearest neighbour embracing graph in Euclidean and hyperbolic space, and \citet{Rosen2025} study the degree of the origin and edge-length functionals of the hyperbolic radial spanning tree; in both cases the hyperbolic results are contrasted with their Euclidean counterparts. The threshold graph processes studied in this paper may be viewed as intrinsic, manifold-valued analogues of these spatial random graph models.

A second line of work relevant to the present paper concerns exchangeability. The probabilistic theory of exchangeable arrays was developed by \citet{Hoover1979, Aldous1981, Kallenberg1989}. In the setting of graphs, this theory is closely related to the graphon framework. Graphons were introduced by \citet{Lovasz2006} as limit objects for convergent sequences of dense graphs, with convergence expressed in terms of subgraph densities. \citet{Borgs2006} and \citet{Borgs2008} placed this theory on a metric footing through the cut metric, while \citet{Lovasz2010} further developed the topological structure of the space of graphons. \citet{Lovasz2012} provides a comprehensive account of the resulting theory of large dense graph limits. \citet{Diaconis2008} made explicit the connection between graph limits and the Aldous--Hoover--Kallenberg representation of jointly exchangeable random arrays. From this perspective, an infinite jointly exchangeable adjacency matrix can be represented via a graphon. This representation is powerful but intrinsically dense: a jointly exchangeable infinite random graph is almost surely either empty or dense; see \citet{OrbanzRoy2015}.

The graphex framework provides one of the principal resolutions of this tension between exchangeability and sparsity. Building on the representation theory for exchangeable random measures due to \citet{Kallenberg1990}, \citet{CaronFox2017} and \citet{VeitchRoy2015} constructed sparse random graphs from exchangeable random measures on $\rbb_{+}^{2}$. \citet{VeitchRoy2015} identified the graphex as the canonical representation for such models and established integrability conditions ensuring local finiteness. Subsequent work developed a corresponding sparse graph limit theory: \citet{Borgs2018} introduced graphon processes and studied convergence, \citet{Janson2022} developed graphex convergence, \citet{BorgsChayes2019} studied sampling consistency, and \citet{Veitch2019} addressed estimation. Different sparsity regimes are considered by \citet{Caron2023}. The present paper is complementary to this literature. Rather than imposing exchangeability on an edge measure and discarding isolated vertices, we instead ask how classical matrix exchangeability may be recovered on bounded restrictions of the underlying point process.

Another related area is latent-space network modelling. In the model of \citet{Hoff2002}, vertices are assigned unobserved positions in a Euclidean space and connection probabilities depend on latent distances or projections. This approach was extended by \citet{Handcock2007} to include clustering and by \citet{Salter2017} to multiview network data. The stochastic blockmodel of \citet{Holland1983} may be viewed as a latent-space model with a discrete latent space. Non-Euclidean latent spaces have also received substantial attention. Hyperbolic latent-space models were proposed by \citet{Krioukov2010} and \citet{Boguna2010}, motivated by the emergence of heterogeneous degrees and strong clustering in negatively curved spaces. Rigorous results on degree sequences and clustering in random hyperbolic graphs were obtained by \citet{Gugelmann2012}. Further work on hyperbolic network geometry includes the geometric renormalisation procedure of \citet{Garcia2018}. Spherical latent spaces have been studied, for example, by \citet{McCormick2015} in the context of aggregated relational data. For surveys on geometric and latent-space network models, see \citet{Smith2019} and \citet{Boguna2021}.

Inference for latent geometries has also been studied extensively. The problem of estimating the curvature of a latent space is addressed by \citet{Lubold2023} and \citet{Wilkins2022}, while the inverse problem of embedding an observed network into a latent space is discussed, for example, by \citet{Raftery2012}, \citet{Asta2024}, and \citet{Rastelli2024}. These works are primarily statistical in nature. By contrast, the present paper focuses on structural and asymptotic properties of graph processes generated from intrinsic Riemannian volume and distance.

The present paper develops a point-process framework for random graphs on general latent spaces and studies the form of matrix exchangeability induced by bounded restrictions. It then specialises this framework to homogeneous Riemannian manifolds and derives the corresponding limiting degree and clustering statistics, including their dependence on the underlying geometry.

\subsection{Outline}

The paper is organised as follows. Section~\ref{SectionGraphProcesses} introduces the graph-process framework: random adjacency measures directed by locally finite point processes and connected by measurable kernels. We discuss i.i.d.\ representations and compare graph processes with the graphon and graphex models, showing in particular that a Poisson graph process is jointly exchangeable as a random measure only when its connection kernel is essentially constant. Section~\ref{SectionExchangeability} introduces \textit{bounded exchangeability}: we show that every bounded restriction of a boundedly exchangeable graph process determines an infinite jointly exchangeable adjacency matrix, unique in distribution, and hence admits a local Aldous--Hoover--Kallenberg representation by a graphon; we also describe how these local graphons transform as the observation window changes.

Section~\ref{SectionRiemannianManifolds} specialises the construction to Riemannian manifolds, using intrinsic distance and volume to define a geometrically natural class of graph processes. For this class, we prove a scale-invariance theorem describing how simultaneous rescaling of the metric, vertex density, and connection rule leaves the induced graph distribution unchanged. The principal example is the threshold model, a manifold-valued analogue of Gilbert's random plane network. Section~\ref{SectionApplications} studies this model on homogeneous infinite-volume manifolds: we establish asymptotic sparsity and derive the limiting degree distribution and clustering coefficient. The asymptotic results for Euclidean and hyperbolic space show that, after matching the limiting degree distribution, the limiting clustering coefficients differ. The section concludes with simulations that illustrate the proved sparsity, degree, and clustering limits and the Euclidean–hyperbolic comparison.

\section{Graph Processes} \label{SectionGraphProcesses}

In this section we set up the graph-process framework used throughout the paper. The vertices are identified with the atoms of a locally finite point process on a Polish latent space $S$, and, conditionally on these latent locations, edges between distinct vertices are drawn independently according to a symmetric connection kernel $W : S^2 \to [0,1]$. In this sense, graph processes extend the conditional-independence structure of the graphon model beyond latent probability spaces, allowing in particular locally finite point processes with infinitely many atoms. We then compare this framework with the graphon and graphex constructions. The graphon model is recovered on probability spaces under i.i.d.\ sampling, whereas the graphex comparison shows that, under the standing assumptions of Section~\ref{SubsectionGraphexComparison}, a Poisson graph process is jointly exchangeable in the random-measure sense only when its connection kernel is essentially constant.

\subsection{Foundations}

Broadly speaking, a graph consists of a set of points together with a set of connections between pairs of these points. The points are called vertices or nodes, and the connections are called edges. We assume throughout that edges are undirected and unweighted, and no two vertices are joined by more than one edge. Formally, a graph is a pair $G=(v(G),e(G))$, where $v(G)$ is its (countable) vertex set and
\[
    e(G)\subseteq \bigl\{\{x,y\} \mid x,y\in v(G)\bigr\}
\]
is its edge set. Thus, edges are regarded as unordered pairs of vertices, and an edge of the form $\{x\}$ represents a loop at $x$. Unless otherwise stated we exclude loops, in which case the graph is called simple. When the vertices are enumerated by integers, a graph may be represented by an adjacency matrix --- a symmetric $\{0,1\}$-valued matrix $A=(A_{ij})_{i,j \geq 1}$ with $A_{ij}=1$ if and only if vertices $i$ and $j$ are adjacent; in particular, $A_{ii} = 0$ for simple graphs. This representation depends on the chosen enumeration.

For a general measure space $(S, \Scal, \mu)$ and $B \in \Scal$, we let $\mu\vert_B$ denote the restriction of the measure $\mu$ to $B$, that is, $\mu\vert_B(C) \coloneqq \mu(B \cap C)$ for all $C \in \Scal$. Following \citet{Kallenberg2017}, we work with a fixed complete separable metric space $(S, d)$ and write $\Bcal(S)$ for its Borel $\sigma$-algebra. We denote by $\widehat{\Bcal}(S)$ the ring of Borel sets that are bounded with respect to $d$. A measure $\mu$ on $S$ is locally finite if $\mu(B) < \infty$ for every $B \in \widehat{\Bcal}(S)$; we write $\Mcal(S)$ for the space of locally finite measures on $S$, a Polish space under the corresponding vague topology. We call $\mu \in \Mcal(S)$ a simple measure if it is a countable sum of distinct unit point masses.

All random elements in this paper are defined on a common probability space $(\Omega, \Fcal, \pbb)$, assumed rich enough to support the independent $\mathrm{U}[0,1]$ random variables used below. When needed, we use \citet[Corollary~8.18]{Kallenberg2021} to realise distributionally specified random elements on this probability space so that the corresponding representation holds almost surely. Unless stated otherwise, $U_{\emptyset}$, $(U_i)$ and $(U_{ij})$ denote independent arrays of i.i.d.\ $\mathrm{U}[0,1]$ random variables with $U_{ij} = U_{ji}$, independent of all other randomness. We write $X \indep Y$ when the random elements $X$ and $Y$ are independent. Finally, we define $\nbb_0 \coloneqq \nbb \cup \{0\}$, $\overline{\nbb} \coloneqq \nbb \cup \{\infty\}$, $\overline{\nbb}_0 \coloneqq \nbb_0 \cup
\{\infty\}$, and $\rbb_+ \coloneqq [0, \infty)$. The terminology below closely follows \citet{VeitchRoy2015}.
\begin{definition} \label{DefAdjMeasure}
An adjacency measure on $S$ is a simple measure $\xi$ on $S^2\times\{0,1\}$ such that
\begin{enumerate}[label=(\roman*)]
    \item there is a simple measure $\nu$ on $S$ with $\xi(C\times\{0,1\}) = \nu^{\otimes 2}(C)$ for every $C \in \Bcal(S)^{\otimes 2}$,
    \item $(\theta \otimes \mathrm{id})\,\xi = \xi$, where $\theta(x,y)=(y,x)$ and $\mathrm{id}$ denotes the identity function on $\{0,1\}$.
\end{enumerate}
For $B\in\mathcal{B}(S)$, the $B$-restriction of $\xi$ is the adjacency measure $\xi\vert_{B \times B \times \{0,1\}}$, which we denote by $\xi\vert_B$ for brevity.
\end{definition}
Here, local finiteness on $S^2 \times \{0,1\}$ is understood with respect to any $p$-product metric of the fixed metric on $S$ and the discrete metric on $\{0,1\}$, the choice being immaterial.

Let $\Acal(S)$ denote the class of adjacency measures on $S$. The set of atoms of $\xi \in \Acal(S)$ is
\[
\mathrm{At}(\xi)
\coloneqq
\bigl\{(x,y,a)\in S^2\times\{0,1\}:
\xi(\{(x,y,a)\})=1
\bigr\}.
\]
Projecting onto the first coordinate, we obtain $V(\xi) \coloneqq \pi_1(\mathrm{At}(\xi))$, the \textit{vertex set} of $\xi$. Observe that $\nu$ in Definition~\ref{DefAdjMeasure} is uniquely determined and equal to
\[
    \nu_{\xi} \coloneqq \sum_{x \in V(\xi)} \delta_x.
\]
We then say that $\xi$ is \textit{directed} by $\nu_{\xi}$; we simply write $\nu$ when $\xi$ can be inferred from the context. By Definition~\ref{DefAdjMeasure}, for each $x,y \in V(\xi)$ there is exactly one $a_{\xi}(x,y) \in \{0,1\}$ such that $(x,y,a_{\xi}(x,y)) \in \mathrm{At}(\xi)$. An enumeration $V(\xi) = \{s_i \mid i \leq n\}$ with $n \in \overline{\nbb}_0$, together with $a_{ij} \coloneqq a_{\xi}(s_i, s_j)$, yields
\[
\xi = \sum_{i,j \leq n} \delta_{(s_i, s_j, a_{ij})}.
\]
We call the $\{0,1\}$-valued symmetric array $(a_{ij})_{i,j \le n}$ the \textit{adjacency matrix} of the representation; it depends on the chosen enumeration of $V(\xi)$.

We now formalise a correspondence between adjacency measures and graphs.
\begin{definition} \label{DefAdjMeasureOfSimpleGraph}
    Let $G$ be an undirected, unweighted graph without multiple edges, possibly with loops, whose vertex set $v(G)\subseteq S$ is locally finite. The adjacency measure $\xi_G$ associated with $G$ is given by 
    \[
        \xi_G \coloneqq \sum_{x,y\in v(G)}
        \delta_{(x,y,\mathbf{1}\{\{x,y\}\in e(G)\})}.
    \]
\end{definition}
It is readily confirmed that this is indeed an adjacency measure. Contrary to the convention of \citet{VeitchRoy2015}, our definition retains vertices of degree zero. Consequently, two graphs $G$ and $G'$ have the same adjacency measure if and only if they have the same vertex and edge sets, hence are equal. The associated graph of an adjacency measure, recovering $G$ from its measure, is made precise in Definition~\ref{DefGraphAssociatedWithAdjMeasure}.
\begin{definition} \label{DefGraphAssociatedWithAdjMeasure}
    Let $\xi$ be an adjacency measure. The graph $G_\xi$ associated with $\xi$ has vertex set $v(G_\xi) \coloneqq V(\xi)$ and edge set
    \[
        e(G_{\xi}) \coloneqq \bigl\{ \{ x,y\} \mid x,y \in V(\xi) \ \text{and} \ a_{\xi}(x,y) = 1 \bigr\}.
    \]
\end{definition}
We now introduce randomness.
\begin{definition}
    A random adjacency measure on $S$ is a random measure $\xi$ on $S^2\times\{0,1\}$ such that $\pbb(\xi\in\Acal(S))=1$.
\end{definition}
By \citet[Lemma~1.6]{Kallenberg2017}, there exists a random variable $N \in \overline{\nbb}_0$ and a.s.\ distinct random elements $(S_i)_{i \leq N}$ of $S$ such that the point process $\nu_{\xi}$ on $S$ satisfies
\[
\nu_{\xi} \overset{a.s.}{=} \sum_{i \leq N} \delta_{S_i}.
\]
Setting $A_{ij} \coloneqq a_{\xi}(S_i, S_j)$, we obtain
\begin{equation} \label{EqRandomAdjacencyMeasure}
    \xi \overset{a.s.}{=} \sum_{i,j \le N} \delta_{(S_i, S_j, A_{ij})}.
\end{equation}
Conversely, any locally finite sum \eqref{EqRandomAdjacencyMeasure}, involving a random variable $N \in \overline{\nbb}_0$, a.s.\ distinct random elements $(S_i)_{i \le N}$ of $S$, and a $\{0,1\}$-valued symmetric random array $A \coloneqq (A_{ij})_{i,j \le N}$, is a random adjacency measure.

Our proposed random graph model rests on the correspondence between graphs and adjacency measures: its vertices are the atoms of a point process, and its edges are drawn conditionally independently according to a connection kernel, made precise in Definition~\ref{DefGraphProcess} below.
\begin{definition}
    A connection kernel on $S$ is a symmetric measurable function $W : S^2 \to [0,1]$.
\end{definition}
A connection kernel on $[0,1]$ is called a \textit{graphon} in the graph-limit literature, see for example \citet{Borgs2008}. To reflect the interplay of graphs and point processes, we call our random graph model a \textit{graph process}.
\begin{definition} \label{DefGraphProcess}
    A random adjacency measure $\xi$ on $S$ is a graph process on $S$ if it admits a representation \eqref{EqRandomAdjacencyMeasure} in which, for some connection kernel $W$ on $S$,
    \begin{equation*}
        A_{ii} = 0, \qquad
        A_{ij}\mid \bigl( N,(S_k) \bigr) \;\overset{\textnormal{ind.}}{\sim}\; \mathrm{Bernoulli}\bigl(W(S_i,S_j)\bigr), \qquad 
        i < j,
    \end{equation*}
    the remaining entries being given by symmetry, $A_{ji}=A_{ij}$. We call \eqref{EqRandomAdjacencyMeasure} an i.i.d.\ representation of $\xi$ if, in addition, the $(S_k)$ are i.i.d.\ according to some probability measure $\hat\mu$ on $S$ and $N\indep(S_k)$.
\end{definition}
We then say that $\xi$ is \textit{connected} by $W$. Unlike the directing random measure $\nu$, the connecting kernel $W$ is not uniquely determined by $\xi$: several kernels may connect the same graph process. This non-uniqueness is revisited in Section~\ref{SectionExchangeability}, where the local graphons associated with bounded restrictions are shown to be unique only up to measure-preserving transformations. Note also that every graph process has a simple associated graph, irrespective of the values of $W$ on the diagonal of $S^2$.

A point process $\nu$ on $S$ with $\nu \overset{d}{=} \sum_{i \le N} \delta_{S_i}$, for a random variable $N \in \nbb_0$ and an array $(S_i) \indep N$ of i.i.d.\ random elements of $S$ with common distribution $\hat{\mu}$, is a \textit{mixed binomial process} based on $N$ and $\hat{\mu}$; note that $N < \infty$ a.s., so such a process is a.s.\ finite. It is a \textit{Poisson process} based on $\mu \in \Mcal(S)$, the \textit{intensity measure} of $\nu$, if $\nu(B_1), \dots, \nu(B_n)$ are independent and Poisson distributed with means $\mu(B_1), \dots, \mu(B_n)$ for all disjoint $B_1, \dots, B_n \in \widehat{\Bcal}(S)$; when $\mu(S) = \infty$ this process has infinitely many atoms, though it remains locally finite. It is a \textit{mixed Poisson process} based on a random variable $\delta \ge 0$ and $\mu$ if, conditionally on $\delta$, it is Poisson based on $\delta\mu$. A graph process directed by a (mixed) Poisson or mixed binomial process is called a \textit{(mixed) Poisson} or \textit{mixed binomial graph process}, respectively. Finally, a mixed Poisson process based on $\mu$, respectively a mixed binomial process based on $\hat\mu$, is simple if and only if $\mu$, respectively $\hat\mu$, is diffuse \citep[Lemma~3.6]{Kallenberg2017}.

All measures in this paper are locally finite. When the latent space has infinite measure, we therefore work with bounded restrictions, on which the directing process is a.s.\ finite. Such restrictions remain graph processes; in fact this holds for every Borel set.
\begin{proposition} \label{PropRestrictionOfGraphProcess}
    Let $\xi$ be a graph process on $S$ and $B \in \Bcal(S)$. If $\xi$ is directed by $\nu$ and connected by $W$, then $\xi\vert_B$ is again a graph process, directed by $\nu\vert_B$ and connected by $W$.
\end{proposition}
\begin{proof}
    Let $\xi$ be represented by \eqref{EqRandomAdjacencyMeasure}. Restriction to $B$ retains exactly the atoms indexed by $I_B \coloneqq \{\, i \le N : S_i \in B \,\}$, so $\xi\vert_B$ is a random adjacency measure directed by $\nu\vert_B = \sum_{i \in I_B} \delta_{S_i}$. Since $I_B$ is a deterministic function of $N$ and $(S_i)$, conditioning on $N$ and $(S_i)$ leaves the retained off-diagonal marks $(A_{ij})_{i,j \in I_B, i<j}$ independent with $A_{ij} \sim \operatorname{Bernoulli}(W(S_i,S_j))$, while $A_{ii}=0$; hence $\xi\vert_B$ is connected by $W$.
\end{proof}

In Section~\ref{SectionExchangeability} we introduce a local notion of exchangeability, resting on graph processes whose vertices are locally i.i.d.\ samples from a probability measure. It is motivated by the following two lemmas.
\begin{lemma}[{\citealp[Theorem~3.4]{Kallenberg2017}}] \label{LemmaRestrictionMixedPoisson}
    Let $\nu$ be a mixed Poisson or binomial process on $S$ based on $\mu$, and fix any $B \in \widehat{\Bcal}(S)$. Then $\nu\vert_B$ is a mixed binomial process based on $\mu\vert_B$.
\end{lemma}
In a slight abuse of notation, the base $\mu\vert_B$ is understood to be normalised to a probability measure.
\begin{lemma}[{\citealp[Theorem~3.7]{Kallenberg2017}}] \label{LemmaExtensionMixedBinomial}
    Let $\nu$ be a point process on $S$, and fix any $B_1, B_2, \dots \in \widehat{\Bcal}(S)$ with $B_n \uparrow S$. Then $\nu$ is a mixed Poisson or binomial process if and only if $\nu\vert_{B_n}$ is a mixed binomial process for every $n \in \nbb$.
\end{lemma}
These let us establish the local existence of i.i.d.\ representations for mixed Poisson and mixed binomial graph processes.

\begin{proposition} \label{PropLocalIidRep}
    The $B$-restriction of a graph process $\xi$ on $S$ admits an i.i.d.\ representation for every $B \in \widehat{\Bcal}(S)$ if and only if $\xi$ is directed by a mixed Poisson or mixed binomial point process.
\end{proposition}
\begin{proof}
    Let $\xi$ be directed by $\nu$; by Proposition~\ref{PropRestrictionOfGraphProcess}, $\xi\vert_B$ is directed by $\nu\vert_B$. Suppose $\xi\vert_B$ admits an i.i.d.\ representation $\xi\vert_B \overset{a.s.}{=} \sum_{i,j \le N} \delta_{(S_i,S_j,A_{ij})}$. Since $B$ is bounded, $N = \nu(B) < \infty$ a.s., and by Definition~\ref{DefGraphProcess} the locations $(S_i)$ are i.i.d.\ with $N \indep (S_i)$; hence $\nu\vert_B \overset{a.s.}{=} \sum_{i \le N} \delta_{S_i}$ is a mixed binomial process. Fix bounded sets $B_n \uparrow S$; the same argument applied to each $B_n$ shows that $\nu\vert_{B_n}$ is mixed binomial for every $n$, so by Lemma~\ref{LemmaExtensionMixedBinomial} $\nu$ is a mixed Poisson or mixed binomial process.

    Conversely, suppose $\nu$ is mixed Poisson or mixed binomial. By Proposition~\ref{PropRestrictionOfGraphProcess} and Lemma~\ref{LemmaRestrictionMixedPoisson}, $\xi\vert_B$ is directed by the mixed binomial process $\nu\vert_B \overset{a.s.}{=} \sum_{i \le N} \delta_{S_i}$, with $(S_i) \indep N$ i.i.d.\ from some probability law on $B$. Since $\xi$ is a graph process it is connected by some kernel $W$, and by Proposition~\ref{PropRestrictionOfGraphProcess} so is $\xi\vert_B$. 
    Define 
    \[
        \widetilde A_{ii} \coloneqq 0, \qquad
        \widetilde A_{ij} \coloneqq \mathbf 1\{U_{ij}\leq W(S_i,S_j)\}, \qquad 
        i<j,
    \]
    the remaining entries being given by symmetry, $\widetilde A_{ji} = \widetilde A_{ij}$, and set $\widetilde\xi\coloneqq\sum_{i,j\leq N}\delta_{(S_i,S_j,\widetilde A_{ij})}$. Then $\widetilde{\xi}$ is a graph process with an i.i.d.\ representation, and $\widetilde{\xi} \overset{d}{=} \xi\vert_B$ by Definition~\ref{DefGraphProcess}. By the remark at the beginning of the section, this distributional identity may be realised almost surely, giving an i.i.d.\ representation of $\xi\vert_B$.
\end{proof}
A local i.i.d.\ representation does not entail a global one. Indeed, let $\xi$ be a Poisson graph process on $S$ directed by $\nu$ with intensity measure $\mu$ satisfying $\mu(S) = \infty$. A global i.i.d.\ representation would make $\nu$ a mixed binomial process, hence a.s.\ finite; but a Poisson process with intensity measure $\mu$ satisfying $\mu(S) = \infty$ has infinitely many atoms almost surely, a contradiction. A unit-rate Poisson graph process on $\mathbb{R}$ is one such example.

\subsection{Comparison to graphon models}
\label{SubsectionGraphonComp}

A random matrix $(X_{ij})_{i,j \ge 1}$ is \textit{jointly exchangeable} if, for all $n \ge 1$ and all permutations $\pi$ of $\{1, \dots, n\}$,
\begin{equation} \label{EqExchangeableArray}
    (X_{ij})_{i,j \le n} \overset{d}{=} (X_{\pi(i)\pi(j)})_{i,j \le n}.
\end{equation}
This symmetry underpins an array analogue of de Finetti's theorem, the Aldous--Hoover--Kallenberg theorem, due to \citet{Hoover1979, Aldous1981, Kallenberg1989}, which provides a functional representation of such arrays.

A random graph is \textit{(jointly) exchangeable} if it admits a jointly exchangeable adjacency matrix. Recall that the adjacency matrix of a random graph is a symmetric $\{0,1\}$-valued random matrix with a vanishing diagonal. The Aldous--Hoover--Kallenberg theorem then asserts that an infinite jointly exchangeable adjacency matrix $A = (A_{ij})_{i,j \in \nbb}$ admits a representation
\begin{equation} \label{EqAldousHooverKallenbergRepresentation}
    A_{ij} \overset{a.s.}{=} \mathbf{1}\{U_{ij} \le w(U_\emptyset, U_i, U_j)\}, \qquad i,j \in \nbb,
\end{equation}
where $w : [0,1]^3 \to [0,1]$ is measurable and symmetric in the second and third argument. Because $A$ has a vanishing diagonal, the representing function may, and henceforth will, be chosen so that
\begin{equation*}
    w(u_0, u, u) = 0 \qquad \text{for all } u_0, u \in [0,1].
\end{equation*}
Conversely, every array with a representation \eqref{EqAldousHooverKallenbergRepresentation} involving such a zero-diagonal representative is a jointly exchangeable adjacency matrix. The array is \textit{ergodic} when $w$ does not depend on $U_\emptyset$, in which case \eqref{EqAldousHooverKallenbergRepresentation} reduces to 
\begin{equation} \label{EqAldousHooverKallenbergRepErgodic}
    A_{ij} \overset{a.s.}{=} \mathbf{1}\{U_{ij} \le W(U_i, U_j)\}, \qquad i,j \in \nbb
\end{equation}
for a graphon $W : [0,1]^2 \to [0,1]$, chosen so that $W(u,u) = 0$ for every $u \in [0,1]$. The general case is a mixture of such ergodic laws.

For the zero-diagonal representative $W$ fixed above, the ergodic case \eqref{EqAldousHooverKallenbergRepErgodic} admits the following sampling description. Sample i.i.d.\ random variables $U_1, U_2, \dots \sim \mathrm{U}[0,1]$. Set $A_{ii}=0$ and draw
\[
    A_{ij} \mid (U_k) \;\overset{\textnormal{ind.}}{\sim}\; \mathrm{Bernoulli}\bigl(W(U_i,U_j)\bigr), \qquad 
    i < j,
\]
the remaining entries being given by symmetry, $A_{ji}=A_{ij}$. The sampled $A$ then has the law of the array in \eqref{EqAldousHooverKallenbergRepErgodic}. We call the associated random graph with vertex set $\nbb$ the \textit{graphon model} based on $W$. We obtain the \textit{finite} graphon model on $n \in \nbb$ vertices based on $W$ by restricting $A$ to its leading principal $n \times n$ submatrix.

The graphon $W$ in \eqref{EqAldousHooverKallenbergRepErgodic} is not uniquely determined by the law of $A$: for example, replacing $W$ by $W \circ (T \otimes T)$ for any measure-preserving transformation $T$ of $[0,1]$ leaves the law of $A$ unchanged. This non-uniqueness is revisited in Section~\ref{SectionExchangeability}.

The following example frames the graphon model as a graph process. Note that only the finite graphon model defines a graph process, for otherwise the directing measure would a.s.\ have infinitely many atoms in $[0,1]$, violating local finiteness. 
\begin{example}
    Let $W : [0,1]^2 \to [0,1]$ be a graphon such that $W(u,u) = 0$ for every $u \in [0,1]$. Sample i.i.d.\ random variables $U_1, U_2,\dots \sim \mathrm{U}[0,1]$. Set $A_{ii} = 0$ and draw
    \begin{equation*}
        A_{ij} \mid (U_k) \;\overset{\textnormal{ind.}}{\sim}\; \mathrm{Bernoulli}\bigl(W(U_i,U_j)\bigr), \qquad 
        i < j,
    \end{equation*}
    the remaining entries being given by symmetry, $A_{ji} = A_{ij}$. For $n \in \nbb$, define a graph process on $[0,1]$ by
    \begin{equation*}
        \xi \coloneqq \sum_{i,j \leq n} \delta_{(U_i,\, U_j,\, A_{ij})}.
    \end{equation*}
    Indeed, this is an i.i.d.\ representation by Definition~\ref{DefGraphProcess}. Its adjacency matrix $(A_{ij})_{i,j \le n}$ is jointly exchangeable and realises the finite graphon model on $n$ vertices based on $W$.
\end{example}
By \citet{OrbanzRoy2015}, an infinite exchangeable graph is almost surely either empty or dense, where \textit{dense} means that the number of edges $e_n$ among the first $n$ vertices satisfies $e_n \ge c\,n^2$ for some constant $c > 0$ independent of $n$. By contrast, we later show that graph processes can produce \textit{sparse} graphs, with $e_n = O(n)$.

\subsection{Comparison to graphex models} \label{SubsectionGraphexComparison}

The \textit{graphex model}, introduced by \citet{CaronFox2017} and \citet{VeitchRoy2015} building on the theory of exchangeable random measures on $\rbb^2_+$ developed by \citet{Kallenberg1990}, generalises the graphon framework to produce sparse infinite graphs. Following \citet{VeitchRoy2015}, the distribution of such a graph is characterised by a \textit{graphex}, a triple $(I, S, W)$ consisting of a nonnegative real $I \in \rbb_+$, an integrable function $S : \rbb_+ \to \rbb_+$, and a symmetric measurable function $W : \rbb_+^2 \to [0,1]$ satisfying weak integrability conditions; these govern, respectively, the isolated edges, the star components, and the principal connection structure of the graph. As \citet{VeitchRoy2015} do, and as later emphasised by \citet{Janson2022}, we restrict to the case $I = S = 0$, in which the graph is determined by $W$ alone; we then refer to $W$ itself as the graphex. In this construction, a unit-rate Poisson process $\Pi = \{(\theta_i, \vartheta_i)\}_{i \ge 1}$ on $\rbb^2_+$ provides the vertices, where $\theta_i \in \rbb_+$ is a time stamp and $\vartheta_i \in \rbb_+$ a latent connectivity parameter. Conditionally on $\Pi$, each pair of distinct vertices is connected independently, i.e.
\begin{equation}
    Z_{ii} = 0, \qquad
    Z_{ij} \mid \Pi \;\overset{\mathrm{ind.}}{\sim}\; \mathrm{Bernoulli}\bigl(W(\vartheta_i, \vartheta_j)\bigr), \qquad i < j,
\end{equation}
the remaining entries being given by symmetry, $Z_{ji} = Z_{ij}$; the edge $\{\theta_i,\theta_j\}$ being present if and only if $Z_{ij} = 1$. The edge random measure $\sum_{Z_{ij}=1} \delta_{(\theta_i,\theta_j)}$ on $\rbb^2_+$ is jointly exchangeable in the sense of \citet{Kallenberg1990}, and is locally finite under the integrability conditions on $W$ established by \citet{VeitchRoy2015}. We emphasise that the adjacency matrix $Z = (Z_{ij})_{i,j \in \nbb}$ need not be jointly exchangeable as an infinite array.

Throughout this subsection we show that graph processes are in general distinct from graphex models, additionally assuming that $S$ is uncountable and endowed with a $\sigma$-finite diffuse Borel measure $\mu$ such that $\mu(S) > 0$. For each $d\in\nbb$, we write $\lambda_d$ for Lebesgue measure on $\rbb^d$, and abbreviate $\lambda \coloneqq \lambda_1$. In particular, $\lambda_2=\lambda^{\otimes2}$.

\begin{definition}
    A measurable map $T : S \to S$ is a \textit{$\mu$-preserving transformation} of $S$ if its pushforward $T\mu \coloneqq \mu \circ T^{-1}$ equals $\mu$, that is, $\mu(T^{-1}(B)) = \mu(B)$ for all $B \in \Bcal(S)$.
\end{definition}

\begin{definition} \label{DefExchangeRandMeasure}
    A random measure $\eta$ on $S^2$ is \textit{jointly $\mu$-exchangeable} if $(T \otimes T)\,\eta \overset{d}{=} \eta$ for every $\mu$-preserving transformation $T$ of $S$, where $T \otimes T : S^2 \to S^2$, $(x,y) \mapsto (Tx, Ty)$, and $(T \otimes T)\eta$ denotes the pushforward of $\eta$.
\end{definition}

Recall that a measurable (with respect to the Borel $\sigma$-algebra) bijection from $S$ to some other topological space $\tilde{S}$ with measurable inverse is called a \textit{Borel isomorphism}. The next result is known, but we could not locate a precise statement in the literature.
\begin{proposition} \label{PropExchangeRandMeasure}
    Let $I \coloneqq [0, \mu(S))$, with the convention $I = \rbb_+$ when $\mu(S) = \infty$. Then there is a Borel isomorphism $\phi : S \to I$ with $\phi\mu = \lambda\vert_I$, and for any such $\phi$, a random measure $\eta$ on $S^2$ is jointly $\mu$-exchangeable if and only if $(\phi \otimes \phi)\eta$ on $I^2$ is jointly $\lambda\vert_I$-exchangeable.
\end{proposition}
We use the following lemma of \citet[Theorem~17.41]{Kechris1995}.
\begin{lemma} \label{LemmaPolishProbSpace}
    If $\mu(S) = 1$, there exists a Borel isomorphism $\phi : S \to [0,1)$ such that $\phi\mu = \lambda\vert_{[0,1)}$.
\end{lemma}
Now we can prove Proposition~\ref{PropExchangeRandMeasure}.
\begin{proof}[Proof of Proposition~\ref{PropExchangeRandMeasure}]
    If $\mu(S) < \infty$, normalising, applying Lemma~\ref{LemmaPolishProbSpace}, and rescaling gives the required Borel isomorphism, so assume $\mu(S) = \infty$. Set $B_0 \coloneqq \emptyset$ and $\beta_0 \coloneqq 0$. Choose $B_n \uparrow S$ in $\Bcal(S)$ with $\beta_n \coloneqq \mu(B_n)$ finite and strictly increasing, and set $D_n \coloneqq B_n \setminus B_{n-1}$. Each increment $D_n$ has mass $\beta_n - \beta_{n-1} > 0$; normalising, applying Lemma~\ref{LemmaPolishProbSpace}, and rescaling yields a Borel isomorphism $\phi_n : D_n \to [\beta_{n-1}, \beta_n)$ with $\phi_n(\mu\vert_{D_n}) = \lambda\vert_{[\beta_{n-1},\beta_n)}$. The increments $D_n$ partition $S$ and their images partition $I \coloneqq \rbb_+$, so the $\phi_n$ glue to a single Borel isomorphism $\phi : S \to I$ such that $\phi\vert_{D_n} = \phi_n$, with $\phi\mu = \lambda\vert_I$.

    It remains to show that $\eta$ is jointly $\mu$-exchangeable if and only if $(\phi \otimes \phi)\eta$ is jointly $\lambda\vert_I$-exchangeable. Suppose the former, and let $T : I \to I$ be $\lambda\vert_I$-preserving. Then $\phi^{-1} \circ T \circ \phi$ is $\mu$-preserving, since
    \[
        (\phi^{-1} \circ T \circ \phi)\mu
        = \phi^{-1}\bigl(T(\phi\mu)\bigr)
        = \phi^{-1}\bigl(T\lambda\vert_I\bigr)
        = \phi^{-1}\lambda\vert_I = \mu.
    \]
    By hypothesis, $\bigl((\phi^{-1}\!\circ T\circ\phi) \otimes (\phi^{-1}\!\circ T\circ\phi)\bigr)\eta \overset{d}{=} \eta$; pushing forward by $\phi \otimes \phi$ and using
    \[
        (\phi \otimes \phi)\circ
        \bigl((\phi^{-1}\!\circ T\circ\phi) \otimes (\phi^{-1}\!\circ T\circ\phi)\bigr)
        = (T \otimes T)\circ(\phi \otimes \phi),
    \]
    we obtain
    $(T \otimes T)(\phi \otimes \phi)\eta \overset{d}{=} (\phi \otimes \phi)\eta$. The
    converse is analogous, exchanging the roles of $\phi$ and $\phi^{-1}$.
\end{proof}

In the graphex literature, a random graph is identified with a random measure on $\rbb^2_+$ whose coordinates are vertex labels and whose point masses mark the edges, following \citet{CaronFox2017, VeitchRoy2015, Borgs2018}. Under the graphex construction this edge measure is jointly $\lambda$-exchangeable in the sense of Definition~\ref{DefExchangeRandMeasure} with $S = \rbb_+$ and $\mu = \lambda$, as established by \citet{Kallenberg1990, VeitchRoy2015}. In contrast to Definition~\ref{DefAdjMeasureOfSimpleGraph}, this encoding records only present edges, so isolated (degree-zero) vertices leave no trace. Adopting the same convention on a general latent space, we define the \textit{edge measure} of a random adjacency measure $\xi$ on $S$ by
\[
    E_{\xi}(C) \coloneqq \xi(C \times \{1\}), \quad C \in \Bcal(S)^{\otimes 2}.
\]
Under a measurable enumeration \eqref{EqRandomAdjacencyMeasure}, this random measure on $S^2$ has the representation
\begin{equation} \label{EqRandomAdjacencyMeasure2}
    E_{\xi} \overset{a.s.}{=} \sum_{i,j \le N} A_{ij}\, \delta_{(S_i, S_j)}.
\end{equation}
Abusing notation, we write $\xi$ for $E_{\xi}$ when the intended interpretation is clear from the context.

For a graph process in particular, the edge measure is not in general jointly $\mu$-exchangeable, unlike the graphex edge measure, as Theorem~\ref{ThmExchangeableGraphProcessIsTrivial} makes precise.
\begin{theorem} \label{ThmExchangeableGraphProcessIsTrivial}
    Let $S$ be uncountable, and let $\xi$ be the edge measure \eqref{EqRandomAdjacencyMeasure2} of a Poisson graph process on $S$ with diffuse $\sigma$-finite intensity measure $\mu$, connected by a kernel $W$. Then $\xi$ is jointly $\mu$-exchangeable if and only if $W$ is $\mu^{\otimes 2}$-a.e.\ constant.
\end{theorem}
To prove this, the following lemma is useful.

\begin{lemma} \label{LemInvarianceMeasurePresTrans}
Let $W:[0,1]^2\to[0,1]$ be measurable, and suppose that for every $\lambda\vert_{[0,1]}$-preserving bijection $T$ of $[0,1]$ the identity $W\circ(T\otimes T)=W$ holds $\lambda^{\otimes2}$-a.e. Then $W$ is $\lambda^{\otimes2}$-a.e.\ constant.
\end{lemma}
We defer the short proof of this Lemma to the appendix.

\begin{proof}[Proof of Theorem~\ref{ThmExchangeableGraphProcessIsTrivial}]
If $W$ is $\mu^{\otimes2}$-a.e.\ constant the claim is immediate. Conversely, suppose $\xi$ is jointly $\mu$-exchangeable. By Proposition~\ref{PropExchangeRandMeasure} we may transport $\xi$ along a measure-preserving Borel isomorphism $\phi$ and assume that $\xi$ is a jointly $\lambda\vert_I$-exchangeable Poisson graph process on $I=[0,\mu(S))$ with intensity measure $\lambda\vert_I$ connected by $W:I^2\to[0,1]$. It then suffices to show that $W$ is $\lambda^{\otimes2}$-a.e.\ constant on $I^2$.

Fix a bounded subinterval $J\subseteq I$ and a $\lambda\vert_J$-preserving bijection $T$ of $J$, extended by the identity to a $\lambda\vert_I$-preserving bijection of $I$. For bounded measurable $f,g$ supported in $J$, the multivariate Mecke equation \citep[Theorem~4.4]{Penrose2017} gives
\[
  \ebb\!\left[\int f(x)g(y)\,\xi(dx,dy)\right] 
  =\int_{J^2} f(x)g(y)\,W(x,y)\,dx\,dy,
\]
using that the diagonal adjacency marks of $\xi$ are zero by Definition~\ref{DefGraphProcess}; the right-hand side is finite as $\lambda(J)<\infty$ and $W\le 1$. By joint exchangeability $(T\otimes T)\xi\overset{d}{=}\xi$, so
\[
  \ebb\!\left[\int f(x)g(y)\,\xi(dx,dy)\right]
  =\ebb\!\left[\int f(x)g(y)\,(T\otimes T)\xi(dx,dy)\right]
  =\ebb\!\left[\int f(Tx)g(Ty)\,\xi(dx,dy)\right].
\]
As $f\circ T$ and $g\circ T$ are again supported in $J$, the Mecke identity applies to the right-hand side and yields
\[
  \int_{J^2} f(x)g(y)\,W(x,y)\,dx\,dy
  =\int_{J^2} f(Tx)g(Ty)\,W(x,y)\,dx\,dy.
\]
Composing the right-hand side with $T^{-1}\otimes T^{-1}$, which also preserves $\lambda^{\otimes2}$, gives
\[
  \int_{J^2} f(x)g(y)\,W(x,y)\,dx\,dy
  =\int_{J^2} f(x)g(y)\,W\bigl(T^{-1}x,T^{-1}y\bigr)\,dx\,dy
\]
for all such $f,g$, and so
\[
  W=W\circ\bigl(T^{-1}\otimes T^{-1}\bigr)\qquad\lambda^{\otimes2}\text{-a.e.\ on }J^2.
\]
Rescaling $J$ to $[0,1]$ and applying Lemma~\ref{LemInvarianceMeasurePresTrans}, the restriction of $W$ to $J^2$ equals a constant $c_J$ $\lambda^{\otimes2}$-a.e. Choosing $J_n\uparrow I$ with $\lambda(J_n)>0$, the overlaps $J_n^2\subseteq J_{n+1}^2$ have positive measure, so all $c_{J_n}$ coincide in a common value $c$, giving $W=c$ $\lambda^{\otimes2}$-a.e.\ on $I^2$.
\end{proof}

Note that $\xi$ is jointly $\mu$-exchangeable if and only if it is jointly $\delta\mu$-exchangeable for every $\delta>0$, since a map preserves $\mu$ exactly when it preserves $\delta\mu$. The following example illustrates the essential idea behind the proof of Theorem~\ref{ThmExchangeableGraphProcessIsTrivial}, namely that joint exchangeability fails when the invariance under measure-preserving transformations does not pass to the connection kernel $W$.
\begin{example}
    Let $\xi$ be the edge measure \eqref{EqRandomAdjacencyMeasure2} of a graph process on $\rbb$ directed by a unit-rate Poisson point process and connected by $W:(x,y)\mapsto\mathbf{1}\{\norm{x-y}\le1\}$. Then $\xi$ is not jointly $\lambda$-exchangeable as a random measure on $\rbb^2$. To see this, let $T$ be the $\lambda$-preserving transformation that swaps the intervals $(1,2]$ and $(4,5]$. Then $\ebb[\xi((1,2]\times(2,3])]>0$, whereas $(T\otimes T)\xi\bigl((1,2]\times(2,3]\bigr)=\xi((4,5]\times(2,3])\overset{a.s.}{=}0$, since every $x\in(4,5]$ and $y\in(2,3]$ satisfy $\norm{x-y}>1$.
\end{example}

\section{Exchangeability} \label{SectionExchangeability}

Recall that the class of infinite random graphs with jointly exchangeable adjacency matrices is completely characterised by mixtures of graphon models. In this section we introduce a new notion of exchangeability for graph processes, which we call \textit{bounded exchangeability}. For bounded $B\subseteq S$, we show that under a mild support condition on the directing process, made precise in Definition~\ref{DefBoundedExchangeability}, the restriction $\xi\vert_B$ to vertices in $B$ determines an infinite jointly exchangeable matrix $A^B$, unique in distribution. This holds even when the adjacency matrix of the global graph is not jointly exchangeable. As a consequence, boundedly exchangeable graph processes can locally be represented as graphon models. This involves a local graphon $W_B:[0,1]^2\to[0,1]$, which, up to measure-preserving transformations, is the composition of the zero-diagonal modification 
\begin{equation*}
    \widetilde{W}(x,y) \coloneqq \mathbf{1}\{x \neq y\} W(x,y), \qquad \text{for all } x,y \in S
\end{equation*}
of the connection kernel $W : S^2 \to [0,1]$ with a local sampling function $f_B:[0,1]\to B$. Finally, we establish relations between the local graphons as the observation window changes.

\subsection{Boundedly exchangeable graph processes}

For a real-valued matrix $A=(A_{ij})_{i,j \ge 1}$ and a permutation $\pi$ of $\nbb$, define the permuted matrix $A\circ\pi\coloneqq(A_{\pi(i)\pi(j)})_{i,j \ge 1}$. For $n\in\nbb$, let $A[n]\coloneqq(A_{ij})_{i,j\le n}$ denote the leading principal $n\times n$ submatrix of $A$. We call a permutation $\pi$ of $\nbb$ \textit{finitely supported} if $\pi(i)=i$ for all but finitely many $i\in\nbb$.

First, we prove a lemma concerning random matrices of random finite dimension, establishing necessary and sufficient conditions under which the dimension-conditional distributions arise as the laws of the leading principal submatrices of an infinite jointly exchangeable matrix, which is then unique in distribution. Consider the space of finite-dimensional real-valued square matrices $\Rcal\coloneqq\bigsqcup_{n=0}^{\infty} \rbb^{n\times n}$. A random element $\mathbf A=(A,N)$ in $\Rcal$ consists of a random variable $N\in\nbb_0$ and a random matrix $A\in\rbb^{N\times N}$.
\begin{lemma} \label{LemLifting}
    Let $\mathbf{A} = (A, N)$ be a random element in $\Rcal$ such that $\pbb(N = n) > 0$ for every $n \in \nbb$. Then $\mathbf{A}$ satisfies
    \begin{enumerate}[label=(\roman*)]
        \item (Finite exchangeability) For all $n \geq 1$ and all permutations $\pi$ of $\{1,\dots,n\}$:
        \begin{equation*}
            A \mid (N = n) \overset{d}{=} A \circ \pi \mid (N = n)
        \end{equation*}
        \item (Compatibility) For all $n \geq 1$ and all $m \geq n$:
        \begin{equation*}
            A \mid (N = n) \overset{d}{=} A[n] \mid (N = m)
        \end{equation*}
    \end{enumerate}
    if and only if there exist an infinite jointly exchangeable matrix $A^{(\infty)} \in \rbb^{\nbb \times \nbb}$ such that for all $n \geq 1$:
    \begin{equation} \label{EqLiftingLemma}
        A \mid (N = n) \overset{d}{=} A^{(\infty)}[n]
    \end{equation}
    In that case, $A^{(\infty)}$ is unique in distribution. 
\end{lemma}
\begin{proof}
    The converse is trivial, so we only prove the forward direction. By the compatibility of conditional laws and full support on $\nbb$, Kolmogorov's extension theorem applies. Hence, there exists a random matrix $A^{(\infty)} \in \rbb^{\mathbb{N} \times \mathbb{N}}$, which is unique in distribution, such that \eqref{EqLiftingLemma} holds for all $n \geq 1$. Let $\pi$ be a finitely supported permutation of $\nbb$ and choose $n_0$ with $\pi(i)=i$ for all $i>n_0$. For every $n\ge n_0$ the permutation $\pi$ restricts to a permutation of $\{1,\dots,n\}$, so
    \begin{equation*}
        \left( A^{(\infty)} \circ \pi \right)[n] = \left(A^{(\infty)}[n] \right) \circ \pi \overset{d}{=} A^{(\infty)}[n],
    \end{equation*}
    the equality of laws holding by \eqref{EqLiftingLemma} together with condition~(i). Since this holds for all $n \ge n_0$ and $\pi$ is arbitrary, it follows that $A^{(\infty)}$ is jointly exchangeable.
\end{proof} 

The idea is to apply Lemma~\ref{LemLifting} to the adjacency matrix of a graph process. For the adjacency matrix of an i.i.d.\ representation, the first two conditions follow from the i.i.d.\ structure of the locations and the conditional independence of the marks; in light of Proposition~\ref{PropLocalIidRep} we therefore make the following definition.
\begin{definition} \label{DefBoundedExchangeability}
    We call a graph process $\xi$ on $S$ \textit{boundedly exchangeable} if it is directed by a mixed Poisson or mixed binomial process $\nu$ and, for every $B\in\widehat{\Bcal}(S)$ such that $\pbb(\nu(B)=0)<1$, we have $\pbb(\nu(B)=n)>0$ for all $n\in\nbb$.
\end{definition}
Note that $\xi$ is automatically boundedly exchangeable when $\nu$ is a mixed Poisson process, since the count $\nu(B)$ then charges every $n\in\nbb$ whenever $\pbb(\nu(B)=0)<1$. For a mixed binomial process this need not hold, which is why the support condition is imposed.

\begin{theorem} \label{ThmUniqueLocalInfiniteExchangeableMatrix}
    Let $\xi$ be a boundedly exchangeable graph process on $S$ directed by $\nu$. For every $B \in \widehat{\Bcal}(S)$ where $\pbb(\nu(B) = 0) < 1$ there exists an infinite jointly exchangeable matrix $A^B \in \rbb^{\nbb \times \nbb}$ such that the adjacency matrix $(A, N)$ of any i.i.d.\ representation of $\xi\vert_B$ satisfies for all $n \geq 1$:
    \begin{equation} \label{EqUniqueLocalInfiniteExchangeableMatrix}
        A \mid (N = n) \overset{d}{=} A^B[n]
    \end{equation}
    Moreover, the distribution of $A^B$ is unique.
\end{theorem}

\begin{proof}
Let $B \in \widehat{\Bcal}(S)$ with $\pbb(\nu(B) = 0) < 1$. By bounded exchangeability, $\pbb(\nu(B)=n)>0$ for every $n\in\nbb$, so the conditional laws below are well defined. By Proposition~\ref{PropLocalIidRep}, $\xi\vert_B$ admits an i.i.d.\ representation
\[
  \xi\vert_B \overset{a.s.}{=} \sum_{i,j \leq N} \delta_{(S_i,S_j,A_{ij})},
\]
where $(S_i)$ are i.i.d.\ random elements of $B$, independent of $N \overset{a.s.}{=} \nu(B)$. If $W : S^2 \to [0,1]$ is a connection kernel for this representation, then
\[
  A \mid (N = n) \overset{d}{=} \bigl(\mathbf{1}\{i\ne j\} \mathbf{1}\{U_{ij}\le W(S_i,S_j)\}\bigr)_{i,j\le n}.
\]
Thus Lemma~\ref{LemLifting} applies and yields an infinite jointly exchangeable matrix $A^B$ satisfying \eqref{EqUniqueLocalInfiniteExchangeableMatrix}. It may be realized as
\[
  A^B_{ij} \coloneqq \mathbf{1}\{i\ne j\} \mathbf{1}\{U_{ij} \le W(S_i,S_j)\}, \qquad i,j \in \nbb.
\]

It remains to show that the law of $A^B$ is independent of the chosen i.i.d.\ representation. Fix $n \in \nbb$ and work conditionally on $\{N=n\}$. Denote by $\Acal_n(B)$ the standard Borel space of adjacency measures on $B$ with $n$ vertices, let
\begin{align*}
    \Ycal_n &\coloneqq \bigl\{(s_1,\dots,s_n) \in B^n \mid s_i \neq s_j \text{ for } i \neq j \bigr\} \times 
    \bigl\{(a_{ij})_{i,j \leq n} \in \{0,1\}^{n \times n} \mid a_{ii} = 0, a_{ij} = a_{ji} \bigr\},
\end{align*}
also standard Borel, and define the measurable map
\begin{equation*}
    \Phi_n : \Ycal_n \to \Acal_n(B), \qquad 
    \Phi_n \bigl( (s_i)_{i \le n}, (a_{ij})_{i,j \le n}\big) \coloneqq \sum_{i,j \le n} \delta_{(s_i, s_j, a_{ij})}.
\end{equation*}
The group $\mathfrak{S}_n$ of permutations of $\{1, \dots, n\}$ acts on $\Ycal_n$ by
\begin{equation*}
    \pi \bigl( (s_i)_{i \le n}, (a_{ij})_{i,j \le n} \bigr) \coloneqq \bigl( (s_{\pi(i)})_{i \le n}, (a_{\pi(i)\pi(j)})_{i,j \le n} \bigr).
\end{equation*}
The labelled pair $L_n \coloneqq \bigl( (S_i)_{i\leq n}, (A_{ij})_{i,j \le n} \bigr)$ is invariant in distribution under this action, $\Phi_n$ is invariant, and each fibre $\Phi_n^{-1}\{\zeta\}$ consists of the $n!$ enumerations of the vertex set of $\zeta \in \Acal_n(B)$. The law $\mu_n$ of $L_n$ pushes forward under $\Phi_n$ to $\nu_n \coloneqq \Lcal(\xi\vert_B \mid N = n)$. 

By \citet[Theorem~8.5]{Kallenberg2021}, there exists a probability kernel $K_n$ from $\Acal_n(B)$ to $\Ycal_n$ such that 
\begin{enumerate}[label=(\roman*)]
    \item $K_n(\zeta, \Phi^{-1}_n(\zeta)) = 1$ for $\nu_n$-a.e. $\zeta \in \Acal_n(B)$,
    \item $\mu_n(E) = \int_{\Acal_n(B)} K_n(\zeta, E)\,\nu_n(d\zeta)$ for every $E \in \Bcal(\Ycal_n)$,
\end{enumerate}
and any two probability kernels satisfying (i)--(ii) agree for $\nu_n$-a.e. $\zeta$. Note that, for every $\pi\in\mathfrak{S}_n$,
\begin{equation*}
    K_n^\pi(\zeta, E)
    \coloneqq
    K_n(\zeta, \pi^{-1}E),
\end{equation*}
is another probability kernel satisfying (i)--(ii); this follows from finite exchangeability of $L_n$ and invariance of $\Phi_n$. By uniqueness, $K^{\pi}_n(\zeta, \cdot) = K_n(\zeta, \cdot)$ for $\nu_n$-a.e. $\zeta$. Since $K_n(\zeta, \cdot)$ is a probability measure on the finite set $\Phi^{-1}_n(\zeta)$ for $\nu_n$-a.e. $\zeta$, and $\mathfrak{S}_n$ acts transitively on this set, we conclude that $K_n(\zeta, \cdot)$ is the uniform measure on the fibre $\Phi^{-1}_n(\zeta)$ for $\nu_n$-a.e. $\zeta$. 

If $\bigl( (S'_i)_{i\leq n}, (A'_{ij})_{i,j \leq n} \bigr)$ is obtained from another i.i.d.\ representation, then $N' \overset{a.s.}{=} N$ and the same argument yields a probability kernel $K'_n$ from $\Acal_n(B)$ to $\Ycal_n$ such that $K'_n(\zeta, \cdot)$ is the uniform measure on the fibre $\Phi^{-1}_n(\zeta)$ for $\nu_n$-a.e. $\zeta$. Integrating both with respect to $\nu_n$, (ii) yields
\begin{equation*}
    A \mid (N = n) \overset{d}{=} A' \mid (N' = n)
\end{equation*}
for every $n \in \nbb$. By the uniqueness clause of Lemma~\ref{LemLifting}, the law of $A^B$ is the same for both representations.
\end{proof}
We call the unique distribution determined by \eqref{EqUniqueLocalInfiniteExchangeableMatrix} the \textit{(jointly) exchangeable distribution associated with $\xi$ on $B$}. Any random matrix having this distribution is called a \textit{(jointly) exchangeable matrix associated with $\xi$ on $B$}. In particular, every associated matrix $A^B$ has a vanishing diagonal almost surely; equivalently,
\[
\mathbb{P}\bigl(A^B_{ii}=0\text{ for all }i\in\mathbb{N}\bigr)=1.
\]

\subsection{Aldous--Hoover--Kallenberg representations}

Recall that infinite jointly exchangeable arrays are characterised by their Aldous--Hoover--Kallenberg representation, which we now discuss in relation to boundedly exchangeable graph processes; we will see that the bounded restrictions realize the ergodic case of this representation. Throughout this subsection, let $\xi$ be a boundedly exchangeable graph process on $S$ directed by $\nu$ and connected by a connection kernel $W:S^2\to[0,1]$. Also, let $\widetilde{W}(x,y) \coloneqq \mathbf{1}\{x \neq y\}W(x,y)$ for every $x,y \in S$ be the zero-diagonal modification of $W$. Thus $\widetilde W=W$ off the diagonal of $S^2$, while $\widetilde W(x,x)=0$ for every $x \in S$.

We use the following transfer lemma.
\begin{lemma}[{\citealp[Lemma~4.22]{Kallenberg2021}}] \label{LemTransfer}
    Let $K$ be a probability kernel from a measurable space $X$ to a standard Borel space $Y$. Then there exists a measurable function $f : X \times [0,1] \to Y$ such that if $U$ is $\mathrm{U}[0,1]$, then $f(x, U)$ has distribution $K(x, \cdot)$ for every $x \in X$.
\end{lemma}

First, we describe a localization procedure by which the zero-diagonal modification $\widetilde W$ of the connection kernel, composed with a local sampling function, gives rise to local graphons, recovering the classical ergodic graphon model on each bounded region.
\begin{proposition}\label{PropLocalGraphon}
    For any $B\in\widehat{\Bcal}(S)$ with $\pbb(\nu(B)=0)<1$, there exists a graphon $W_B : [0,1]^2 \to [0,1]$ such that every exchangeable matrix $A^B$ associated with $\xi$ on $B$ admits a representation
    \begin{equation}\label{EqAldousHooverKallenbergLocalGraphon}
        A^B_{ij} \overset{a.s.}{=} \mathbf{1}\{U_{ij}\le W_B(U_i,U_j)\}, \qquad i,j \in \nbb.
    \end{equation}
\end{proposition}
\begin{proof}
    By the construction in Theorem~\ref{ThmUniqueLocalInfiniteExchangeableMatrix} there is an array $(S^B_i)$ of i.i.d.\ random elements of $B$ such that the matrix defined by
    \[
        A^B_{ij} \coloneqq \mathbf{1}\{i\ne j\} \mathbf{1}\bigl\{U_{ij}\le W(S^B_i,S^B_j)\bigr\}, \qquad i,j \in \nbb,
    \]
    is an exchangeable matrix associated with $\xi$ on $B$. Since the $S^B_i$ are almost surely distinct, this matrix satisfies
    \[
        A^B_{ij} \overset{a.s.}{=} \mathbf{1}\bigl\{U_{ij}\le \widetilde{W}(S^B_i,S^B_j)\bigr\}, \qquad i,j \in \nbb.
    \]
    As a Borel subset of the Polish space $S$, the set $B$ is a standard Borel space. If $\eta_B$ is the marginal law of the $(S^B_i)$, then by Lemma~\ref{LemTransfer} there exists a measurable $f_B:[0,1]\to B$ with $f_B(U)\sim\eta_B$ for $U\sim\mathrm U[0,1]$; hence there exists a representation $(S^B_i)\overset{a.s.}{=}f_B(U_i)$. Substituting this yields
    \[
        A^B_{ij} \overset{a.s.}{=} \mathbf{1}\bigl\{U_{ij}\le \widetilde{W}(f_B(U_i),f_B(U_j))\bigr\}, \qquad i,j \in \nbb,
    \]
    which is \eqref{EqAldousHooverKallenbergLocalGraphon} with $W_B \coloneqq \widetilde{W} \circ(f_B\otimes f_B)$. Indeed, $W_B$ is a graphon, being symmetric and measurable as $W$ is. Since $\tilde A^B\overset{d}{=}A^B$ for any other exchangeable matrix $\tilde A^B$ associated with $\xi$ on $B$, the same representation holds for $\tilde A^B$.
\end{proof}
We call $W_B$ a \textit{local graphon} for $\xi$ on $B$. Observe that \eqref{EqAldousHooverKallenbergLocalGraphon} is the ergodic Aldous--Hoover--Kallenberg representation of an infinite jointly exchangeable adjacency matrix; in other words, the exchangeable distribution associated with $\xi$ on $B$ is a graphon model based on $W_B$. Similarly, we call $f_B$ a \textit{sampling function} for $\xi$ on $B$. Its defining property, that $f_B(U)\sim\eta_B$ for $U\sim\mathrm U[0,1]$, depends only on the marginal law $\eta_B$ of the $(S^B_i)$ and hence not on the connection kernel $W$; in particular a single sampling function localizes every kernel connecting $\xi$. The function $f_B$ itself is generally not unique.

Like the sampling functions, local graphons are generally not unique; however, any two are related through measure-preserving transformations.
\begin{proposition} \label{PropLocalGraphonUniqueness}
    For any $B \in \widehat{\Bcal}(S)$ with $\pbb(\nu(B) = 0) < 1$, let $W_B$ and $W'_B$ be two local graphons for $\xi$ on $B$. Then there exist $\lambda\vert_{[0,1]}$-preserving transformations $T, T' : [0,1] \to [0,1]$ such that
    \begin{equation} \label{EqLocalGraphonEquivalence}
        W_B \circ (T \otimes T) = W'_B \circ (T' \otimes T') \quad \lambda^{\otimes 2}\text{-a.e.\ }
    \end{equation}
\end{proposition}
\begin{proof}
    By definition, the infinite jointly exchangeable matrices generated by $W_B$ and $W'_B$ via \eqref{EqAldousHooverKallenbergLocalGraphon} are equal in distribution. Applying \citet[Theorem~7.28]{Kallenberg2005} to their off-diagonal laws yields \eqref{EqLocalGraphonEquivalence}.
\end{proof}

To understand how local graphons and exchangeable distributions behave under changing frames of reference, it suffices to find suitable sampling functions. The following two results outline how sampling functions for different frames of reference may be obtained from one another via conditional sampling.
\begin{proposition} \label{PropSampFuncLargeFromSmall}
    Let $B, C \in \widehat{\Bcal}(S)$ be disjoint such that $\pbb(\nu(B) = 0) < 1$ and $\pbb(\nu(C) = 0) < 1$, and let $f_B$ and $f_C$ be sampling functions for $\xi$ on $B$ and $C$, respectively. Then, for some $p \in (0,1)$,
    \begin{align*}
        f_{B \cup C}(x) \coloneqq \begin{cases}
            f_B \left( \frac{x}{p}\right), &\text{for } x \in [0, p) \\
            f_C \left( \frac{x - p}{1 - p} \right), &\text{for } x \in [p, 1]
        \end{cases}
    \end{align*}
    is a sampling function for $\xi$ on $B \cup C$.
\end{proposition}
\begin{proof}
    Since $B\cup C\in\widehat{\Bcal}(S)$, Proposition~\ref{PropLocalIidRep} gives an i.i.d.\ representation of $\xi\vert_{B\cup C}$; let $(S_i)$ be its vertex array and $\eta$ their marginal law, well defined as $\pbb(\nu(B\cup C)=0)<1$. Note that $\xi\vert_B=(\xi\vert_{B\cup C})\vert_B$, and conditioning the i.i.d.\ vertices on $B$ gives vertex marginal $\eta(B)^{-1}\eta\vert_B=:\eta_B$ for $\xi\vert_B$, which is the law sampled by $f_B$; similarly $\eta_C\coloneqq\eta(C)^{-1}\eta\vert_C$ is sampled by $f_C$. As $B,C$ partition the support of $\eta$, we have the mixture $\eta=\eta(B)\,\eta_B+\eta(C)\,\eta_C$. It suffices to show $f_{B\cup C}(U)\sim\eta$ for $U\sim\mathrm U[0,1]$. Set $p\coloneqq\eta(B)$. Then $p\in(0,1)$: if $p=0$ then $\eta(B)=0$, so $\xi\vert_B$ would have no vertices a.s., contradicting $\pbb(\nu(B)=0)<1$; the case $p=1$ contradicts $\pbb(\nu(C)=0)<1$ symmetrically. The maps $x\mapsto x/p$ and $x\mapsto(x-p)/(1-p)$ send $[0,p)$ and $[p,1]$ uniformly onto $[0,1]$, so $f_{B\cup C}(U)=f_B(U/p)\sim\eta_B$ on $\{U<p\}$ and $f_{B\cup C}(U)=f_C((U-p)/(1-p))\sim\eta_C$ on $\{U\ge p\}$. Hence $f_{B\cup C}(U)\sim p\,\eta_B+(1-p)\eta_C=\eta$.
\end{proof}

\begin{proposition} \label{PropSampFuncSmallFromLarge}
    Let $B, C \in \widehat{\Bcal}(S)$ be disjoint such that $\pbb(\nu(B) = 0) < 1$, and let $f_{B \cup C}$ be sampling function for $\xi$ on $B \cup C$. Then, defining $T_B : [0,1] \to [0,1]$ by
    \begin{equation*}
        T_B(x) \coloneqq \inf \left\{ y \in [0,1] \,\middle\vert\, \frac{\lambda\left( f^{-1}_{B \cup C}(B) \cap [0,y] \right)}{\lambda\left( f^{-1}_{B \cup C}(B) \right)} \geq x \right\},
    \end{equation*}
    the composition $f_{B \cup C} \circ T_B$ is a sampling function for $\xi$ on $B$.
\end{proposition}
\begin{proof}
    Let $\eta$ be the law of $f_{B\cup C}(U)$ for $U\sim\mathrm U[0,1]$; then $f_{B\cup C}$ pushes $\lambda\vert_{[0,1]}$ to $\eta$, so $\lambda(f^{-1}_{B\cup C}(B))=\eta(B)>0$, the positivity from $\pbb(\nu(B)=0)<1$. A map $f_B$ is a sampling function for $\xi$ on $B$ if and only if $f_B(U)\sim\eta(B)^{-1}\eta\vert_B$. Now $T_B$ is the quantile function of the uniform law on $f^{-1}_{B\cup C}(B)$, so $T_B(U)$ is uniform on $f^{-1}_{B\cup C}(B)$; pushing this forward by $f_{B\cup C}$ yields $\eta$ conditioned on $B$, namely $\eta(B)^{-1}\eta\vert_B$. Hence $f_{B\cup C}\circ T_B$ is a sampling function for $\xi$ on $B$.
\end{proof}

Next, we show that a jointly exchangeable matrix $A^{B \cup C}$ associated with $\xi\vert_{B \cup C}$ is a composition of two jointly exchangeable matrices, $A^B$ and $A^C$, associated with $\xi\vert_B$ and $\xi\vert_C$, respectively, and a separately exchangeable (see \citet{Kallenberg2005}) matrix $A^{BC}$ describing the interactions. 
\begin{proposition}
    Let $B, C \in \widehat{\Bcal}(S)$ be disjoint such that $\pbb(\nu(B) = 0) < 1$ and $\pbb(\nu(C) = 0) < 1$, and let $A^{B \cup C}$ be a jointly exchangeable matrix associated with $\xi$ on $B \cup C$. Then there exist a random partition of $\nbb$ into subsequences $(b_i)$ and $(c_i)$ such that a.s.\
    \begin{align*}
        A^B_{ij} &\coloneqq A^{B \cup C}_{b_ib_j}, \\
        A^C_{ij} &\coloneqq A^{B \cup C}_{c_ic_j}
    \end{align*}
    are jointly exchangeable matrices associated with $\xi$ on $B$ and $C$, respectively. Moreover, the matrix $A^{BC}_{ij} \coloneqq A^{B \cup C}_{b_ic_j}$ is separately exchangeable.
\end{proposition}
\begin{proof}
    Let $f_{B\cup C}$ be a sampling function for $\xi$ on $B\cup C$. Then there exists a representation 
    \[
        A^{B\cup C}_{ij} \overset{a.s.}{=} \mathbf{1}\bigl\{U_{ij}\le \widetilde{W}(f_{B \cup C}(U_i), f_{B \cup C}(U_j))\bigr\}, \qquad i,j \in \nbb.
    \]
    For $k\in\nbb$, let $b_k$ be the index of the $k$-th element of $(U_i)$ lying in $f^{-1}_{B\cup C}(B)$, and define $c_k$ analogously with $f^{-1}_{B\cup C}(C)$. Since $f_{B\cup C}$ maps into $B\cup C$, these preimages partition $[0,1]$ up to a null set, so $(b_i)$ and $(c_i)$ partition $\nbb$ and are each a.s.\ infinite. The subsequence of $(U_i)$ in $f^{-1}_{B\cup C}(B)$ is i.i.d.\ uniform on $f^{-1}_{B\cup C}(B)$, so $(U_{b_i})\overset{d}{=}(T_B(U_i))$; as $(U_{ij})\perp(U_i)$ and $(b_i)$ is $(U_i)$-measurable, the marks $(U_{b_ib_j})$ remain i.i.d.\ uniform. Hence, Proposition~\ref{PropSampFuncSmallFromLarge} shows that
    \[
        A^B_{ij} \overset{a.s.}{=} \mathbf{1}\bigl\{U_{b_ib_j}\le \widetilde{W}(f_{B\cup C}(U_{b_i}),f_{B\cup C}(U_{b_j}))\bigr\}, \qquad i,j \in \nbb
    \]
    is a jointly exchangeable matrix associated with $\xi$ on $B$; the argument for $A^C$ is analogous. Finally, $(b_i)$ and $(c_j)$ select disjoint, hence independent, sub-arrays of $(U_i)$, so
    \[
        (f_{B\cup C}(U_{b_i}),f_{B\cup C}(U_{c_j}))\overset{d}{=}(T_B(U_i),T_C(U'_j))
    \]
    for an independent i.i.d.\ array $(U'_j)$; as $(U_{b_ic_j})$ is independent of both groups, $A^{BC}$ is separately exchangeable.
\end{proof}
For more background on separately exchangeable arrays, see \citet{Kallenberg2005}.

\section{Graph Processes on Riemannian Manifolds} \label{SectionRiemannianManifolds}

We now specialise to graph processes on a connected, geodesically complete $d$-dimensional Riemannian manifold $(M,g)$, on which the Riemannian volume measure $\lambda_g$ and the Riemannian distance $d_g$ furnish natural notions of volume and distance. These single out a geometrically natural class of graph processes: the vertices form a Poisson process with intensity measure $\delta\lambda_g$, where $\delta>0$, and the connection kernel is distance-based,
\[
    W(x,y)=h(d_g(x,y)), \quad x,y \in M,
\]
where $h \colon \rbb_+ \to [0,1]$ is measurable and non-increasing. This models the assumption that nearby distinct vertices are more likely to be connected, and yields a manifold-valued analogue of the random connection model \citep{MeesterPenrose1997}. For this class we prove a scale-invariance theorem (Theorem~\ref{ThmScaleInvariance}): for every $c > 0$, the simultaneous rescaling
\[
  g' = c^2 g, \qquad \delta' = c^{-d}\delta, \qquad h'(r) = h(r/c)
\]
leaves the law of the induced graph process invariant. This theorem lets us normalise one of the parameters without loss of generality. The special case 
\[
    h(r) = \mathbf{1}_{[0,t]}(r), \quad r \in \rbb_+,
\]
is called the \textit{threshold model} on $M$ with density $\delta > 0$ and threshold $t > 0$. Here, two distinct vertices are adjacent exactly when $d_g(x,y) \le t$; this is the manifold-valued analogue of the random plane network of \citet{Gilbert1961} and the object of study in the remaining sections. 

We begin with a few general observations. As is standard, manifolds are taken to be second-countable and Hausdorff. Any two points of a connected smooth manifold can be joined by a piecewise regular curve of finite length, so for $x, y \in M$ the \textit{Riemannian distance}
\[
  d_g(x, y) \coloneqq \inf\bigl\{ \operatorname{length}(\gamma) :
    \gamma \text{ a piecewise regular curve from } x \text{ to } y \bigr\}
\]
on $M$ induced by $g$ is finite and well-defined. Indeed, $d_g$ is a metric on $M$ whose induced topology coincides with the manifold topology. Being second-countable, $M$ is separable, so $(M, d_g)$ is a separable metric space; geodesic completeness, via the Hopf--Rinow theorem, renders it complete and yields that closed and $d_g$-bounded subsets of $M$ are compact. Hence $(M, d_g)$ is a complete separable metric space, and we may take $\widehat{\Bcal}(M)$ to be the ring of $d_g$-bounded Borel sets. For details we refer to \citet{Lee2018}. 

The next lemma constructs the \textit{Riemannian volume measure} $\lambda_g$ on $M$, the intrinsic volume measure induced by $g$, which we use to define uniform distributions on Borel subsets of $M$ of finite volume. These in turn single out a geometrically natural class of graph processes on $M$.
\begin{lemma} \label{LemRiemannLebesgueMeasure}
    There exists a unique Borel measure $\lambda_g$ on $M$ such that for every chart $(V, \varphi)$ of $M$, with $\varphi = (x^1, \dots, x^d)$, and every $B \in \Bcal(M)$ with $B \subseteq V$,
    \begin{equation} \label{EqRiemannLebesgueMeasure}
        \lambda_g(B) = \int_{\varphi(B)} \varphi_{\ast}\sqrt{\abs{G}} \; d\lambda_d,
    \end{equation}
    where $G \coloneqq \det[g_{jk}]$ is the determinant of the metric components $g_{jk} \coloneqq g\!\left(\frac{\partial}{\partial x^j}, \frac{\partial}{\partial x^k}\right)$ in the chart $(V, \varphi)$, and $\varphi_{\ast}\sqrt{\abs{G}} = \sqrt{\abs{G}} \circ \varphi^{-1}$ denotes the pushforward of $\sqrt{\abs{G}}$ on $V$ along $\varphi$. Moreover, $\lambda_g$ is diffuse, $\sigma$-finite, and satisfies $\lambda_g(B) < \infty$ for every $B \in \widehat{\Bcal}(M)$.
\end{lemma}
\begin{proof}
    Existence and uniqueness of $\lambda_g$ are proved in Chapter XII.1 of \citet{Escher2008}. By Theorem~XII.1.7(iii), $\lambda_g$ is diffuse. By XII.1.7(ii), $\lambda_g$ is Radon, so every $x \in M$ has an open neighborhood $V_x$ with $\lambda_g(V_x) < \infty$. If $K \subseteq M$ is compact, finitely many $V_x$ cover $K$, so $\lambda_g(K) < \infty$. By XII.1.7(i), $M$ is $\sigma$-compact, so $\lambda_g$ is $\sigma$-finite. Finally, if $B \in \widehat{\Bcal}(M)$, then $B$ is $d_g$-bounded, hence so is $\overline{B}$; being closed and bounded, $\overline{B}$ is compact by Hopf--Rinow, and therefore $\lambda_g(B) \le \lambda_g(\overline{B}) < \infty$.
\end{proof}
For $B \in \Bcal(M)$ with $0 < \lambda_g(B) < \infty$, we call a random element of $M$ with distribution $\lambda_g(B)^{-1}\,\lambda_g\vert_B$ \textit{$g$-uniformly distributed} on $B$. For details we refer to \citet{Escher2008}.

As described at the beginning of the section, let $\xi$ be a Poisson graph process on $M$ with intensity measure $\delta\lambda_g$ for some $\delta > 0$, connected by the kernel $W \colon (x,y) \mapsto h(d_g(x,y))$. This setup possesses a scale invariance, which we now make precise.
\begin{theorem}[Scale invariance] \label{ThmScaleInvariance}
    Let $\xi$ be the Poisson graph process defined above. For any $c > 0$, let $\xi'$ be obtained from $\xi$ by the transformation
    \[
        g' = c^2 g, \qquad \delta' = c^{-d}\delta, \qquad h'(r) = h(r/c).
    \]
    Then $\xi' \overset{d}{=} \xi$.
\end{theorem}
\begin{proof}
    Consider $B \in \widehat{\Bcal}(M)$. By Proposition~\ref{PropRestrictionOfGraphProcess} and Lemma~\ref{LemmaRestrictionMixedPoisson}, $\xi\vert_B$ is directed by $\sum_{i \leq N} \delta_{X_i}$ with $N \sim \operatorname{Poisson}(\delta\lambda_g(B))$ and $(X_i) \indep N$ i.i.d.\ $g$-uniform on $B$, and is connected by $W = h \circ d_g$; likewise $\xi'\vert_B$ is directed by $\sum_{i \leq N'} \delta_{X'_i}$ with $N' \sim \operatorname{Poisson}(\delta'\lambda_{g'}(B))$ and $(X'_i) \indep N'$ i.i.d.\ $g'$-uniform points, and
    connected by $h' \circ d_{g'}$.
    
    Since $g' = c^2 g$, the metric components scale as $g'_{jk} = c^2 g_{jk}$, so the determinant in Lemma~\ref{LemRiemannLebesgueMeasure} satisfies $G' = c^{2d} G$ and $\sqrt{\abs{G'}} = c^d \sqrt{\abs{G}}$. Hence, by \eqref{EqRiemannLebesgueMeasure}, $\lambda_{g'} = c^d \lambda_g$ on any chart $(V, \varphi)$, and hence on $\Bcal(M)$. In particular the $g$- and $g'$-uniform laws on $B$ coincide, and
    \[
        \delta' \lambda_{g'}(B) = c^{-d}\delta \cdot c^d \lambda_g(B) = \delta \lambda_g(B),
    \]
    so the two directing measures share their count and location laws, hence agree in distribution.

    For the kernels, $g' = c^2 g$ scales lengths by $c$, so $d_{g'} = c\, d_g$ and
    \[
        h' \circ d_{g'} = h'(c\, d_g) = h(d_g) = h \circ d_g
    \]
    by $h'(r) = h(r/c)$. Thus $\xi\vert_B$ and $\xi'\vert_B$ share their directing law and connection kernel, so $\xi\vert_B \overset{d}{=} \xi'\vert_B$. As $B \in \widehat{\Bcal}(M)$ was arbitrary, $\xi' \overset{d}{=} \xi$.
\end{proof}

For the remainder of the paper we focus on the threshold model on $(M,g)$. Because its edges are determined by the vertex locations, much of the subsequent analysis simplifies. The threshold model also shows that boundedly exchangeable graph processes need not be globally exchangeable, which is the content of the next result.
\begin{proposition} \label{PropThresholdModelNotExchangeable}
    Let $(M,g)$ satisfy $\lambda_g(M) = \infty$. A threshold model $\xi$ on $M$ with density $\delta > 0$ and threshold $t > 0$ admits no
    representation~\eqref{EqRandomAdjacencyMeasure} whose adjacency matrix is
    jointly exchangeable.
\end{proposition}
\begin{proof}
    Since $\lambda_g(M) = \infty$, the directing Poisson process has infinitely many atoms a.s., so any representation \eqref{EqRandomAdjacencyMeasure} has an infinite adjacency matrix. Suppose such a representation has $(A_{ij})$ jointly exchangeable, and set $X_j \coloneqq A_{1,\,j+1}$ for $j \ge 1$. Permutations of $\nbb$ fixing $1$ act on $A$ as joint permutations and send the first row to a permutation of itself, so $(X_j)_{j\ge1}$ is an infinite exchangeable $\{0,1\}$-sequence. By de Finetti's theorem there is a random $\Theta \in [0,1]$ such that, conditionally on $\Theta$, the $X_j$ are i.i.d.\ $\mathrm{Bernoulli}(\Theta)$.

    The degree of vertex $1$ is $D_1 = \sum_{j\ge1} X_j$, which is finite a.s.\ by local finiteness. Conditionally on $\Theta = \theta$ with $\theta > 0$, the $X_j$ are i.i.d.\ with $\sum_j \pbb(X_j = 1) = \infty$, so the second Borel--Cantelli lemma gives $D_1 = \infty$ a.s., a contradiction. Hence $\Theta = 0$ a.s., so $\pbb(A_{12} = 1) = 0$; by joint exchangeability every pair maps to $(1,2)$ under a permutation, whence $\pbb(A_{ij} = 1) = 0$ for all $i \ne j$. By countability the graph is a.s.\ edgeless.

    But for any $x_0 \in M$, the ball $B_{t/2}(x_0)$ has positive volume, so with positive probability it contains at least two atoms --- which are then within distance $t$ and hence adjacent. This contradicts a.s.\ edgelessness, so no jointly exchangeable representation exists.
\end{proof}

\section{Applications} \label{SectionApplications}

In this section we treat examples of the threshold model on a homogeneous Riemannian manifold $(M,g)$ with $\lambda_g(M) = \infty$. On such a manifold we first show that the threshold model produces sparse graphs (Proposition~\ref{PropConvSparse}) --- a regime that graphon models, by the dichotomy recalled in Section~\ref{SubsectionGraphonComp}, cannot attain. We then derive the limiting degree distribution and clustering coefficient of the infinite random graph $G$ associated with a threshold model $\xi$. Euclidean and hyperbolic space are subsequently treated in detail: by the Killing--Hopf theorem, the complete, simply connected Riemannian manifolds of constant sectional curvature are, up to isometry and in each dimension, the sphere, Euclidean space, and hyperbolic space, and the infinite-volume assumption excludes the compact spherical case. Our examples concentrate on the $2$-dimensional case, since exact formulas become unwieldy in higher dimensions. In particular, we show that while the degree law is fixed by the expected number of points in a connection ball, curvature shifts the clustering coefficient independently of the degree distribution.

Here, a connected Riemannian manifold is called \textit{homogeneous} if its group of isometries acts transitively on it; this implies geodesic completeness (see, e.g., \cite{Lee2018}), so we remain in the setting of Section~\ref{SectionRiemannianManifolds}. Since isometries preserve $\lambda_g$ and map metric balls onto metric balls, the volume of the closed ball $B_t(x_0) \subseteq M$ centred at $x_0 \in M$ with radius $t > 0$ does not depend on $x_0$, and we write $V_t \coloneqq \lambda_g(B_t(x_0))$. For the same reason, the probability that two independent $g$-uniform points in $B_t(x_0)$ lie at distance at most $t$ from each other is independent of $x_0$; for reasons that will become apparent, we call it the \textit{clustering tendency} of $(M, g)$ at radius $t$, denoted by $p_t$.

Throughout this section, let $\xi$ be a threshold model on $(M,g)$ with density $\delta > 0$ and threshold $t > 0$, and let $G$ be the random graph associated with $\xi$; as before, $v(G)$ and $e(G)$ denote its vertex and edge set. Since $\lambda_g(M) = \infty$, the graph $G$ is almost surely infinite, and it is not immediately obvious how graph statistics such as the degree distribution or the clustering coefficient should be defined for it. One natural approach is to evaluate the statistic of interest along a sequence of finite induced subgraphs of $G$ and to pass to the limit; the limit, however, could a priori depend on the chosen sequence. 

We therefore fix a sequence
$E_1, E_2, \dots \in \Bcal(M)$ with $\lambda_g(E_n) < \infty$ for all $n \in \nbb$ and $E_n \uparrow M$, and show that restricting $G$ to the sets $E_n$ yields limits in probability that do not depend on the choice of this sequence. We may impose the additional growth condition
\begin{equation} \label{CondSummable}
    \sum_{n=1}^{\infty} \frac{1}{\lambda_g(E_n)} < \infty,
\end{equation}
under which our limits hold almost surely rather than merely in probability; it is satisfied, e.g., by balls $E_n = B_n(x_0)$ in $\rbb^d$ for $d \geq 2$.

\subsection{Asymptotic sparsity}

We begin by showing that $G$ is asymptotically sparse: more precisely, the number of edges among the vertices in a growing observation window grows at most linearly in the number of such vertices. The proof relies on the univariate and multivariate Mecke equations \citep[Theorems~4.1 and~4.4]{Penrose2017}.

\begin{proposition}\label{PropConvSparse}
Set $v_n \coloneqq \#\{v(G) \cap E_n\}$ and $e_n \coloneqq \#\bigl\{\{x,y\}\in e(G) \mid x,y\in E_n\bigr\}$. Then, for every $\varepsilon>0$,
\[
    \lim_{n\to\infty}
    \mathbb{P}\left(
        e_n>
        \left(\frac{\delta V_t}{2}+\varepsilon\right)v_n
    \right)
    =0.
\]
If \eqref{CondSummable} holds, then almost surely
\[
  \limsup_{n \to \infty} \frac{e_n}{v_n} \leq \frac{\delta V_t}{2}.
\]
\end{proposition}

\begin{proof}
Let $\xi$ be directed by $\nu=\sum_{i=1}^{\infty}\delta_{S_i}$. The degree of any vertex $S_i$ is $\nu(B_t(S_i))-1$. By the handshaking lemma,
\[
    e_n
    \leq
    \frac{1}{2}
    \sum_{S_i\in E_n}
    \bigl(\nu(B_t(S_i))-1\bigr)
    \eqcolon
    \frac{1}{2}h_n.
\]
By the Mecke equation, in which the atom added at $x$ cancels the $-1$,
\[
    \mathbb{E}[h_n]
    =
    \delta\int_{E_n}
    \mathbb{E}[\nu(B_t(x))]\,\lambda_g(dx)
    =
    \delta^2V_t\lambda_g(E_n),
\]
using that $\nu(B_t(x))\sim\operatorname{Poisson}(\delta V_t)$ for all $x$. Next, the multivariate Mecke equation gives
\begin{align*}
    &\mathbb{E}\Big[
        h_n^2
        -
        \sum_{S_i\in E_n}
        \bigl(\nu(B_t(S_i))-1\bigr)^2
    \Big] \\
    &\quad=
    \mathbb{E}\Big[
        \sum_{\substack{S_i,S_j\in E_n\\S_i\neq S_j}}
        \bigl(\nu(B_t(S_i))-1\bigr)
        \bigl(\nu(B_t(S_j))-1\bigr)
    \Big] \\
    &\quad=
    \delta^2
    \int_{\{d_g(x,y)\leq t\}}
    \mathbb{E}\bigl[
        (\nu(B_t(x))+1)(\nu(B_t(y))+1)
    \bigr]
    \,\lambda_g^{\otimes2}(dx,dy) \\
    &\qquad+
    \delta^2
    \int_{\{d_g(x,y)>t\}}
    \mathbb{E}\bigl[
        \nu(B_t(x))\nu(B_t(y))
    \bigr]
    \,\lambda_g^{\otimes2}(dx,dy) \\
    &\quad=
    \delta^2
    \int_{\{d_g(x,y)>2t\}}
    \mathbb{E}[\nu(B_t(x))]
    \mathbb{E}[\nu(B_t(y))]
    \,\lambda_g^{\otimes2}(dx,dy)
    +
    O(\lambda_g(E_n)) \\
    &\quad=
    \mathbb{E}[h_n]^2+O(\lambda_g(E_n)),
\end{align*}
where all domains of integration are intersected with $E_n^2$. For the third equality, note that $B_t(x)$ and $B_t(y)$ are disjoint when $d_g(x,y)>2t$, so $\nu(B_t(x))$ and $\nu(B_t(y))$ are independent; the complementary region $\{(x,y)\in E_n^2:d_g(x,y)\leq 2t\}$ has $\lambda_g^{\otimes2}$-measure at most $V_{2t}\lambda_g(E_n)$ by homogeneity, while the integrands are bounded by a constant depending only on $\delta V_t$. Since
\[
    \mathbb{E}\Big[
        \sum_{S_i\in E_n}
        \bigl(\nu(B_t(S_i))-1\bigr)^2
    \Big]
    =
    \delta\int_{E_n}
    \mathbb{E}[\nu(B_t(x))^2]\,\lambda_g(dx)
    =
    O(\lambda_g(E_n)),
\]
we obtain $\operatorname{Var}(h_n)=O(\lambda_g(E_n))$. As $E_n \uparrow M$ and $\lambda_g(M) = \infty$, we have $\lambda_g(E_n) \to \infty$, so Chebyshev's inequality yields $h_n/\lambda_g(E_n) \overset{\pbb}{\to} \delta^2 V_t$. Since $v_n = \nu(E_n) \sim \operatorname{Poisson}(\delta\lambda_g(E_n))$, the same argument gives $v_n/\lambda_g(E_n) \overset{\pbb}{\to} \delta$, and hence $\pbb\big(h_n > (\delta V_t + 2\varepsilon) v_n\big) \to 0$ for every $\varepsilon > 0$. As $e_n \leq h_n/2$,
    \begin{equation*}
        \pbb\Big(e_n > \big(\frac{\delta V_t}{2} + \varepsilon\big) v_n\Big)
        \leq \pbb\big(h_n > (\delta V_t + 2\varepsilon)\, v_n\big)
        \longrightarrow 0.
    \end{equation*}

Now suppose that \eqref{CondSummable} holds. The variance estimates above and Chebyshev's inequality imply, for every $\eta>0$,
\[
    \sum_{n=1}^{\infty}
    \mathbb{P}\left(
        \left\vert
            \frac{h_n}{\lambda_g(E_n)}-\delta^2V_t
        \right\vert>\eta
    \right)
    <\infty
    \quad\text{and}\quad
    \sum_{n=1}^{\infty}
    \mathbb{P}\left(
        \left\vert
            \frac{v_n}{\lambda_g(E_n)}-\delta
        \right\vert>\eta
    \right)
    <\infty.
\]
Thus, by the first Borel--Cantelli lemma,
\[
    \frac{h_n}{\lambda_g(E_n)}
    \longrightarrow\delta^2V_t
    \quad\text{and}\quad
    \frac{v_n}{\lambda_g(E_n)}
    \longrightarrow\delta
    \qquad\text{almost surely}.
\]
In particular, $v_n>0$ eventually almost surely, and therefore
\[
    \limsup_{n\to\infty}\frac{e_n}{v_n}
    \leq
    \lim_{n\to\infty}\frac{h_n}{2v_n}
    =
    \frac{\delta V_t}{2}
    \qquad\text{almost surely}.
\]
\end{proof}

\subsection{Asymptotic degree distribution} \label{SubsectionLimitingDegree}

The \textit{degree} of a vertex $x \in v(G)$ is $\deg(x) \coloneqq \#\{y \in v(G) \mid \{x,y\} \in e(G)\}$; in the threshold model, $\deg(x) = \nu(B_t(x)) - 1$ for every vertex $x$, which is almost surely finite since $\nu$ is locally finite.

Using the same technique as in the proof of Proposition~\ref{PropConvSparse}, we now derive the asymptotic degree distribution of $G$.

\begin{proposition} \label{PropConvDegDist}
For $k\in\nbb_0$, write
\[
    \widehat p_n(k)
    \coloneqq
    \frac{1}{\nu(E_n)}
    \sum_{x\in v(G)\cap E_n}
    \mathbf{1}{\{\deg(x)=k\}}
\]
on the event $\{\nu(E_n)>0\}$, and set $\widehat p_n(k) \coloneqq 0$
otherwise. Then
\[
    \widehat p_n(k)
    \overset{\pbb}{\longrightarrow}
    e^{-\delta V_t}\frac{(\delta V_t)^k}{k!}.
\]
If \eqref{CondSummable} holds, then this convergence holds almost surely.
\end{proposition}
\begin{proof}
Let $\xi$ be directed by $\nu=\sum_{i=1}^{\infty}\delta_{S_i}$. A vertex $S_i$ has degree $k$ if and only if $\nu(B_t(S_i))=k+1$, so the number of degree-$k$ vertices in $E_n$ is
\[
    H_n(k)
    \coloneqq
    \sum_{S_i\in E_n}
    \mathbf{1}{\{\nu(B_t(S_i))=k+1\}}.
\]
By the Mecke equation, in which the atom added at $x$ accounts for the shift from $k+1$ to $k$,
\[
    \ebb[H_n(k)]
    =
    \delta\int_{E_n}
    \pbb\bigl(\nu(B_t(x))=k\bigr)\,\lambda_g(dx)
    =
    \delta\lambda_g(E_n)
    e^{-\delta V_t}\frac{(\delta V_t)^k}{k!}.
\]
As in the proof of Proposition~\ref{PropConvSparse}, the multivariate Mecke equation and the bound on the region $\{d_g(x,y)\leq 2t\}$ yield $\operatorname{Var}(H_n(k)) = O\bigl(\lambda_g(E_n)\bigr)$. Hence $H_n(k)/\lambda_g(E_n) \overset{\pbb}{\longrightarrow} \delta e^{-\delta V_t}(\delta V_t)^k/k!$. Since $\nu(E_n) / \lambda_g(E_n) \overset{\pbb}{\longrightarrow}\delta$ and $\pbb\bigl(\nu(E_n)=0\bigr)\longrightarrow0$, the first assertion follows.

Now suppose that \eqref{CondSummable} holds. By the variance bound above, Chebyshev's inequality and the same Borel--Cantelli argument as in the proof of Proposition~\ref{PropConvSparse} yield
\[
    \frac{H_n(k)}{\lambda_g(E_n)}
    \longrightarrow
    \delta e^{-\delta V_t}\frac{(\delta V_t)^k}{k!}
    \quad\text{and}\quad
    \frac{\nu(E_n)}{\lambda_g(E_n)}
    \longrightarrow
    \delta
    \qquad\text{almost surely}.
\]
Hence $\nu(E_n)>0$ eventually almost surely, and therefore
\[
    \widehat p_n(k)
    =
    \frac{H_n(k)/\lambda_g(E_n)}
         {\nu(E_n)/\lambda_g(E_n)}
    \longrightarrow
    e^{-\delta V_t}\frac{(\delta V_t)^k}{k!}
    \qquad\text{almost surely}.
\]
\end{proof}

\begin{corollary}\label{CorEmpiricalDegreeDistribution}
Let
\[
    \widehat{\mu}_n
    \coloneqq
    \sum_{k=0}^{\infty}\widehat p_n(k)\,\delta_k
\]
denote the empirical degree distribution on $\nbb_0$, where $\widehat p_n(k)$ is defined as in Proposition~\ref{PropConvDegDist}, and let $\pi \coloneqq \operatorname{Poisson}(\delta V_t)$.
Then
\[
    d_{\mathrm{TV}}(\widehat{\mu}_n,\pi)
    \xrightarrow{\pbb}0.
\]
If, in addition, \eqref{CondSummable} holds, then this convergence holds almost surely.
\end{corollary}
\begin{proof}
By Proposition~\ref{PropConvDegDist},
\[
    \widehat p_n(k)\xrightarrow{\pbb} 
    \pi(\{k\}) \eqcolon \pi_k
\]
for every $k\in\nbb_0$. Since $\widehat{\mu}_n$ and $\pi$ are probability measures on the countable space $\nbb_0$, pointwise convergence of their masses implies convergence in total variation in probability.

Now suppose that \eqref{CondSummable} holds. Then $\widehat p_n(k) \to \pi_k$ almost surely, separately for every $k\in\nbb_0$. Since $\nbb_0$ is countable, these convergences hold simultaneously outside a single null set. Since $\widehat\mu_n$ is eventually a probability measure almost surely, \citet[Lemma~1.34]{Kallenberg2021} yields
\[
    d_{\mathrm{TV}}\!\left(
        \widehat\mu_n,
        \pi
    \right)
    \longrightarrow 0
    \qquad\text{almost surely}.
\]
\end{proof}

\subsection{Asymptotic clustering coefficient} \label{SubsectionLimitingClustering}

Note that the limiting degree distribution depends on the manifold, the density, and the threshold only through the product $\delta V_t$: on any homogeneous manifold, every Poisson law arises as the limiting degree distribution for a suitable choice of $\delta$. The degree distribution alone therefore cannot distinguish the underlying geometries. This raises the question of whether different manifolds can produce asymptotically different graphs at all; to answer it affirmatively, we turn to the clustering coefficient of $G$.

The \textit{local clustering coefficient} of a vertex $x \in v(G)$ is the proportion of pairs of neighbours of $x$ that are themselves adjacent: writing $\tau(x)$ for the number of triangles in $G$ containing $x$,
\begin{equation*}
    C(x) \coloneqq \frac{2\tau(x)}{\deg(x)(\deg(x)-1)},
\end{equation*}
with the convention that $C(x) \coloneqq 0$ if $\deg(x) \leq 1$. The \textit{clustering coefficient} of a nonempty finite set of vertices is the average of the local clustering coefficients of its members.

\begin{proposition} \label{PropConvClust}
    Write
    \begin{equation*}
        \widehat{C}_n \coloneqq \frac{1}{\nu(E_n)}
        \sum_{x \in v(G) \cap E_n} C(x)
    \end{equation*}
    on the event $\{\nu(E_n) > 0\}$, and $\widehat{C}_n \coloneqq 0$ otherwise. Then
    \begin{equation*}
        \widehat{C}_n \overset{\pbb}{\to}
        \big(1 - e^{-\delta V_t}(1 + \delta V_t)\big)\, p_t.
    \end{equation*}
    If, in addition, \eqref{CondSummable} holds, then this convergence
    holds almost surely.
\end{proposition}

\begin{proof}
Let $\xi$ be directed by $\nu=\sum_{i=1}^{\infty}\delta_{S_i}$, and set
\[
    T_n\coloneqq\sum_{S_i\in E_n}C(S_i).
\]
By homogeneity and the Mecke equation,
\[
    \ebb[T_n]
    =
    \delta\lambda_g(E_n)\ebb^x[C(x)],
\]
where $\ebb^x$ denotes expectation for the process directed by $\nu+\delta_x$, and $x\in M$ is arbitrary. Writing $N_x \coloneqq \nu(B_t(x))$, the multivariate Mecke equation gives
\begin{align*}
    \ebb^x[C(x)]
    &=
    \ebb\Big[
        \sum_{\substack{S_i,S_j\in B_t(x)\\S_i\neq S_j}}
        \frac{\mathbf{1}\{d_g(S_i,S_j)\leq t\}}
             {N_x(N_x-1)}
        \mathbf{1}\{N_x\geq2\}
    \Big] \\
    &=
    \delta^2
    \ebb\Big[\frac{1}{(N_x+2)(N_x+1)}\Big]
    \int_{B_t(x)^2}
    \mathbf{1}\{d_g(y,z)\leq t\}
    \,\lambda_g(dy)\lambda_g(dz).
\end{align*}
Since $N_x\sim\operatorname{Poisson}(\delta V_t)$,
\[
    \ebb\left[\frac{1}{(N_x+2)(N_x+1)}\right]
    =
    \frac{1-e^{-\delta V_t}(1+\delta V_t)}
         {(\delta V_t)^2},
\]
while the integral equals $V_t^2p_t$. Hence, with
\[
    c_t\coloneqq\bigl(1-e^{-\delta V_t}(1+\delta V_t)\bigr)p_t,
\]
we have $\ebb[T_n]=\delta\lambda_g(E_n)c_t$. As in the proof of Proposition~\ref{PropConvSparse}, the multivariate Mecke equation and the bound on the region $\{d_g(x,y)\leq 2t\}$ yield $\operatorname{Var}(T_n)=O(\lambda_g(E_n))$. Thus $T_n/\lambda_g(E_n)\overset{\pbb}{\to}\delta c_t$. Together with $\nu(E_n)/\lambda_g(E_n)\overset{\pbb}{\to}\delta$ and $\pbb(\nu(E_n)=0)\to0$, this gives $\widehat C_n=T_n/\nu(E_n)\overset{\pbb}{\to}c_t$.

If \eqref{CondSummable} holds, the variance bound, Chebyshev's
inequality, and the same Borel--Cantelli argument as in the proof of
Proposition~\ref{PropConvSparse} give
$T_n/\lambda_g(E_n)\to\delta c_t$ almost surely. Since that proof also
gives $\nu(E_n)/\lambda_g(E_n)\to\delta$ almost surely, it follows that
$\widehat C_n\to c_t$ almost surely.
\end{proof}

The limiting average clustering coefficient thus factorizes into $1 - e^{-\delta V_t}(1 + \delta V_t)$ --- the probability that a $\operatorname{Poisson}(\delta V_t)$ variable is at least $2$, and hence a function of the limiting degree distribution alone --- and the purely geometric factor $p_t$. Two manifolds with matched $\delta V_t$ therefore share their limiting degree distribution, yet may differ in their limiting clustering. Subsections~\ref{SubsectionEucSpace} and~\ref{SubsectionHypSpace} show that Euclidean and hyperbolic space do: their clustering tendencies differ at every threshold.

\subsection{Euclidean Space} \label{SubsectionEucSpace}

\begin{table}[ht]
\centering
\footnotesize
\setlength{\tabcolsep}{3pt}
\renewcommand{\arraystretch}{1.25}
\begin{tabular}{@{}p{0.3 \textwidth}p{0.65\textwidth}@{}}
\toprule
Attribute & Threshold model on $\mathbb{R}^d$ \\
\midrule
Distance
&
$d_{\bar g}(x,y)=\lVert x-y\rVert_2$
\\
Euclidean volume element
&
$\lambda_d(dy)=r^{d-1}\,dr\,d\omega$
\\
Ball volume
&
$\displaystyle
\overline V_t^{\,d}
=
\frac{\pi^{d/2}}{\Gamma(d/2+1)}t^d$
\\
Clustering tendency
&
$\displaystyle
\overline p_d
=
1-\frac{3\sqrt3}{4\pi}
\quad\text{for }d=2$
\\
Sparse
&
Yes
\\
Limiting degree law
&
$\operatorname{Poisson}(\delta\overline V_t^{\,d})$
\\
Limiting clustering coefficient
&
$\displaystyle
\bigl(1-e^{-\delta\overline V_t^{\,d}}
(1+\delta\overline V_t^{\,d})\bigr)\overline p_d$
\\
Globally exchangeable
&
No
\\
Boundedly exchangeable
&
Yes
\\
Associated matrix on $B_R(0)$
&
$\displaystyle
A^{B_R(0)}_{ij}
=
\mathbf{1}\{i \neq j\} \mathbf{1}\{\lVert S_i-S_j\rVert_2\leq t\}$,
with $(S_i)$ i.i.d.\ uniform on $B_R(0)$
\\
Sampling function for $B_R(0)$
&
$\displaystyle
f_{B_R(0)}(u)
=
R\sqrt{u_1}
\bigl(\cos(2\pi u_2),\sin(2\pi u_2)\bigr)
\quad\text{for }d=2$
\\
Local graphon for $B_R(0)$
&
$\displaystyle
W_{B_R(0)}(u,v)
=
\mathbf{1}\{u \neq v\} \mathbf{1}\{
\lVert f_{B_R(0)}(u)-f_{B_R(0)}(v)\rVert_2
\leq t
\}$
\\
\bottomrule
\end{tabular}
\caption{Key attributes of the threshold model on $d$-dimensional
Euclidean space with threshold $t>0$ and density $\delta>0$.
Here $u_1$ and $u_2$ are obtained from $u\in[0,1]$ as described below.}
\label{table_euc}
\end{table}

The simplest example of our setting is $d$-dimensional Euclidean space $(\rbb^d, \overline{g})$, where $\overline{g}$ denotes the Euclidean metric; it has constant sectional curvature zero. The Riemannian distance induced by $\overline{g}$ is
\begin{equation*}
    d_{\overline{g}}(x,y) = \norm{x-y}_2
    = \sqrt{\sum_{i=1}^d (x_i - y_i)^2},
    \qquad x, y \in \rbb^d.
\end{equation*}

Since $\sqrt{\vert G \vert} \equiv 1$ in the identity chart, the Riemannian volume measure on $\rbb^d$ induced by $\overline{g}$ is the Lebesgue measure $\lambda_d$. In spherical polar coordinates around the origin (see Appendix~\ref{SubsectionSphericalPolEuc}), the volume element takes the form
\begin{equation*}
    d\lambda_d = r^{d-1}\,dr\,d\omega,
\end{equation*}
where $d\omega$ denotes the surface-area element of the unit sphere $S^{d-1}$, with total mass $A(S^{d-1}) = 2\pi^{d/2}/\Gamma(d/2)$. Integrating, the volume of the closed ball $B_t(x)$ is
\begin{equation*}
    \overline{V}^d_t \coloneqq \frac{A(S^{d-1})}{d}\,t^d
    = \frac{\pi^{d/2}}{\Gamma(d/2+1)}\,t^d,
\end{equation*}
independently of $x$.

By a dilation argument, the clustering tendency of $d$-dimensional Euclidean space at radius $t$ does not depend on $t$; we denote it by $\overline{p}_d$. For $d = 2$, one can show
\begin{equation*}
    \overline{p}_2 = \frac{1}{\pi} \int_0^1
    \Big( 4x\arccos\big(\frac{x}{2}\big) - x^2\sqrt{4 - x^2} \Big)\,dx
    = 1 - \frac{3\sqrt{3}}{4\pi} \approx 0.5865.
\end{equation*}
Consider the threshold model $\xi$ on $\rbb^d$ with density $\delta > 0$ and threshold $t > 0$: points are placed according to a Poisson process with intensity measure $\delta\lambda_d$, and two distinct points are adjacent exactly when their distance is at most $t$. As noted after Proposition~\ref{PropLocalIidRep}, $\xi$ admits no global i.i.d.\ representation, and by Proposition~\ref{PropThresholdModelNotExchangeable} no representation of $\xi$ has a jointly exchangeable adjacency matrix.

However, since its directing process is Poisson, $\xi$ is boundedly exchangeable (see the remark following Definition~\ref{DefBoundedExchangeability}). Hence every closed ball $B_R(x) \subseteq \rbb^d$ with $R > 0$ identifies an infinite jointly exchangeable matrix $A^{B_R(x)}$ via Theorem~\ref{ThmUniqueLocalInfiniteExchangeableMatrix}; the adjacency matrix of any i.i.d.\ representation of $\xi\vert_{B_R(x)}$ may be regarded as a submatrix of $A^{B_R(x)}$ of random finite dimension.
Explicitly, since $\xi$ is connected by the threshold kernel,
\begin{equation*}
    A^{B_R(x)}_{ij} \overset{a.s.}{=}
    \mathbf{1}\{i \neq j\} \mathbf{1}\{ \norm{S_i - S_j} \leq t\}, \qquad i,j \in \nbb,
\end{equation*}
for an array $(S_i)$ of i.i.d.\ uniformly distributed random elements of $B_R(x)$. By homogeneity, the distribution of $A^{B_R(x)}$ does not depend on $x$, and we set $x = 0$ without loss of generality.

By Proposition~\ref{PropLocalGraphon}, $A^{B_R(0)}$ admits an Aldous--Hoover--Kallenberg representation \eqref{EqAldousHooverKallenbergRepErgodic} with a local graphon $W_{B_R(0)} \colon [0,1]^2 \to [0,1]$. To obtain one explicitly, it suffices to choose a sampling function for $\xi$ on $B_R(0)$ --- here, any measurable $f_{B_R(0)} \colon [0,1] \to B_R(0)$ such that $f_{B_R(0)}(U)$ is uniformly distributed on $B_R(0)$ when $U \sim \mathrm{U}[0,1]$ --- and to set
\begin{equation*}
    W_{B_R(0)}(u,v) \coloneqq
    \mathbf{1}\{u \neq v\} \mathbf{1}\big\{ \norm{f_{B_R(0)}(u) - f_{B_R(0)}(v)} \leq t \big\};
\end{equation*}
different choices of the sampling function yield local graphons that are equivalent in the sense of Proposition~\ref{PropLocalGraphonUniqueness}. Sampling functions can be constructed by inverse-transform sampling in suitable coordinates, but the resulting formulas become unwieldy as the dimension increases. We therefore give an explicit construction only for
$d=2$.

Let $U\sim\mathrm{U}[0,1]$ and write
\[
    U=(0.a_1a_2a_3\ldots)_2
\]
for the binary expansion chosen to contain infinitely many zeros. Split
the binary digits into their odd and even subsequences,
\[
    U_1\coloneqq(0.a_1a_3a_5\ldots)_2,
    \qquad
    U_2\coloneqq(0.a_2a_4a_6\ldots)_2.
\]
Then $(U_1,U_2)$ is uniformly distributed on $[0,1]^2$. Consequently,
\[
    f_{B_R(0)}(U)
    \coloneqq
    R\sqrt{U_1}
    \bigl(\cos(2\pi U_2),\sin(2\pi U_2)\bigr)
\]
is uniformly distributed on $B_R(0)$.

For $d = 2$, we illustrate Proposition~\ref{PropConvClust} by simulation. We set $\delta = 1$ and $t = 1$, so that $\delta \overline{V}^2_t = \pi$ and, by Corollary~\ref{CorEmpiricalDegreeDistribution}, the empirical degree distribution converges to the Poisson law with mean $\pi \approx 3.14$. The limiting clustering coefficient of Proposition~\ref{PropConvClust} is $\big(1 - e^{-\pi}(1+\pi)\big)\,\overline{p}_2 \approx 0.482$. Taking $E_n = B_n(0)$, for which $\sum_n \lambda_2(B_n(0))^{-1} = \pi^{-1}\sum_n n^{-2} < \infty$, the convergence is almost sure by \eqref{CondSummable}. We sampled $1000$ independent realizations of $G$ restricted to $B_{31}(0)$ and computed the clustering coefficient $\widehat{C}_R$ of the vertices in $B_R(0)$ for a range of radii $R \leq 30$; this ensures that the local clustering coefficients entering $\widehat{C}_R$ are free of boundary truncation. Figure~\ref{fig:euc_clust} shows the concentration of $\widehat{C}_R$ around the limit, together with a single realization illustrating the almost-sure convergence.
\begin{figure}[h]
\centering
\includegraphics[width=.95\linewidth]{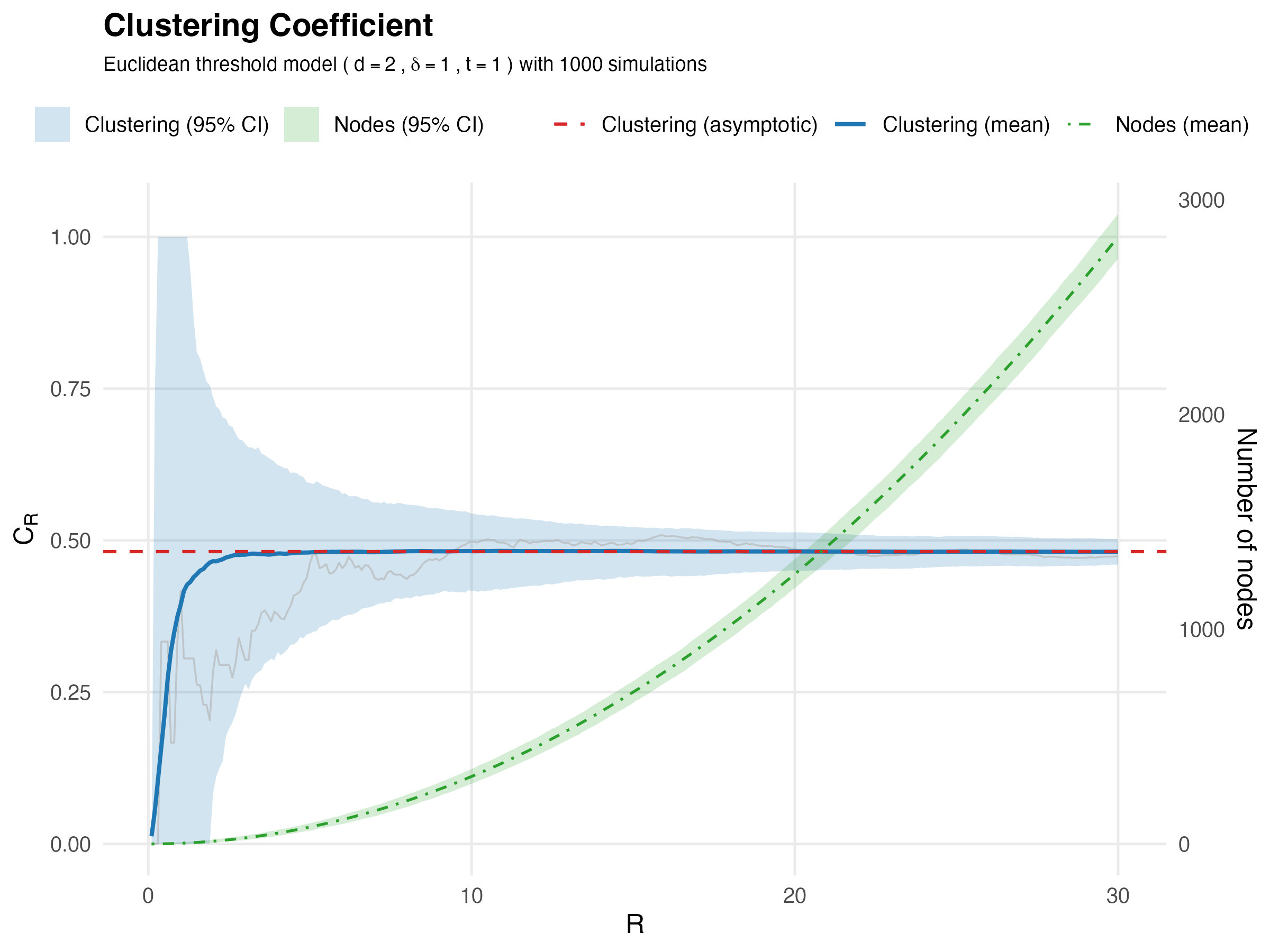}
\caption{Empirical clustering coefficient $\widehat{C}_R$ of the $2$-dimensional Euclidean threshold model with $\delta = 1$ and $t = 1$. Dark blue: mean of $\widehat{C}_R$ over $1000$ realizations; light-blue band: empirical $2.5\%$ and $97.5\%$ quantiles; grey: one randomly selected realization; dashed red: the asymptotic limit $\approx 0.482$ (Proposition~\ref{PropConvClust}). Dark green and light-green band (right axis): mean and corresponding quantiles of the number of vertices in $B_R(0)$.}
\label{fig:euc_clust}
\end{figure}

Both limits above depend on the parameters $(\delta, t)$ only through the product $\delta \overline{V}^2_t$: the limiting degree law is $\operatorname{Poisson}(\delta \overline{V}^2_t)$, and since the clustering tendency $\overline{p}_2$ does not depend on the threshold, the limiting average clustering coefficient is a fixed function of the same quantity. Within Euclidean space, matching the degree distribution therefore also fixes the clustering limit. In hyperbolic space, by contrast, the clustering tendency varies with the curvature, which decouples the two limits; we turn to this next.

\subsection{Hyperbolic Space} \label{SubsectionHypSpace}

\begin{table}[ht]
\centering
\footnotesize
\setlength{\tabcolsep}{3pt}
\renewcommand{\arraystretch}{1.25}
\begin{tabular}{@{}p{0.3\textwidth}p{0.65\textwidth}@{}}
\toprule
Attribute & Threshold model on $\mathbb{H}^d_\kappa$ \\
\midrule
Distance
&
$\displaystyle
d_{\breve g}(x,y)
=
\frac{1}{\kappa}
\operatorname{arcosh}\!\left(
-\kappa^2
\left(-x_0y_0+\sum_{i=1}^d x_i y_i\right)
\right)
$
\\
Hyperbolic volume element
&
$\displaystyle
\lambda_{\breve g}(dy)
=
\frac{1}{\kappa^{d-1}}
\sinh^{d-1}(\kappa r)\,dr\,d\omega
$
\\
Ball volume
&
$\displaystyle
\breve V_t^{\,d}(\kappa)
=
\frac{2\pi^{d/2}}
{\kappa^{d-1}\Gamma(d/2)}
\int_0^t \sinh^{d-1}(\kappa r)\,dr
$
\\
Clustering tendency
&
$\displaystyle
\breve p_d(\kappa t),
\quad\text{see \eqref{EqClustTendHyp} for }d=2
$
\\
Sparse
&
Yes
\\
Limiting degree law
&
$\operatorname{Poisson}(\delta\breve V_t^{\,d}(\kappa))$
\\
Limiting clustering coefficient
&
$\displaystyle
\bigl(
1-e^{-\delta\breve V_t^{\,d}(\kappa)}
(1+\delta\breve V_t^{\,d}(\kappa))
\bigr)\breve p_d(\kappa t)
$
\\
Globally exchangeable
&
No
\\
Boundedly exchangeable
&
Yes
\\
Associated matrix on $B_R(o_\kappa)$
&
$\displaystyle
A^{B_R(o_\kappa)}_{ij}
=
\mathbf{1}\{i \neq j\} \mathbf{1}\{d_{\breve g}(S_i,S_j)\leq t\}$,
with $(S_i)$ i.i.d.\ uniform on $B_R(o_\kappa)$
\\
Sampling function for $B_R(o_\kappa)$
&
$
f_{B_R(o_\kappa)}(u),
\quad\text{see \eqref{EqHypfB} for }d=2
$
\\
Local graphon for $B_R(o_\kappa)$
&
$\displaystyle
W_{B_R(o_\kappa)}(u,v)
=
\mathbf{1}\{u \neq v\}
\mathbf{1}\{
d_{\breve g}(f_{B_R(o_\kappa)}(u),f_{B_R(o_\kappa)}(v))
\leq t
\}
$
\\
\bottomrule
\end{tabular}
\caption{Key attributes of the threshold model on $d$-dimensional
hyperbolic space with curvature parameter $\kappa>0$, threshold $t>0$,
and density $\delta>0$.}
\label{table_hyp}
\end{table}

Our second example is $d$-dimensional hyperbolic space, the complete simply connected manifold of constant negative sectional curvature. By Hilbert's theorem, $2$-dimensional hyperbolic space cannot be isometrically embedded in $3$-dimensional Euclidean space, making it more difficult to visualize. 

For $\kappa > 0$, let $\hbb^d_{\kappa}$ be the set of points $(x_0, x_1, \dots, x_d) \in \rbb^{d+1}$ with $x_0 > 0$ satisfying the relation
\begin{equation*}
    -x_0^2 + \sum_{i = 1}^d x_i^2 = -\frac{1}{\kappa^2},
\end{equation*}
that is, the upper sheet of a two-sheeted hyperboloid in $\rbb^{d+1}$. The pullback of the Minkowski metric $\overline{q} = -\mathrm{d}x_0^2 + \sum_{i = 1}^d \mathrm{d}x_i^2$ on $\rbb^{d+1}$ under the inclusion $\iota \colon \hbb^d_{\kappa} \to \rbb^{d+1}$ is positive definite on $\hbb^d_{\kappa}$ and hence induces a Riemannian metric $\breve{g} \coloneqq \iota^{\ast}\overline{q}$. We call the Riemannian manifold $(\hbb^d_{\kappa}, \breve{g})$ the $d$-dimensional hyperbolic space of radius $1/\kappa$. It is homogeneous with $\lambda_{\breve{g}}(\hbb^d_{\kappa}) = \infty$ and has constant sectional curvature $-\kappa^2$, so it falls within the scope of this section. See \citet{Lee2018} for details.

For $x, y \in \hbb^d_{\kappa}$, the Riemannian distance induced by $\breve{g}$ is given by
\begin{equation*}
    d_{\breve{g}}(x,y)
    = \frac{1}{\kappa}
      \arcosh\left( -\kappa^2 \left(-x_0 y_0
      + \sum_{i = 1}^d x_i y_i\right)\right).
\end{equation*}
Again, it is useful to switch to spherical polar coordinates (see Appendix~\ref{SubsectionSphericalPolHyp}), in which the Riemannian volume measure induced by $\breve{g}$ satisfies
\begin{equation*}
    \mathrm{d}\lambda_{\breve{g}}
    = \frac{1}{\kappa^{d-1}} \sinh^{d-1}(\kappa r)\,\mathrm{d}r\,\mathrm{d}\omega.
\end{equation*}
Integrating, the volume of the closed ball $B_t(x) \subseteq \hbb^d_{\kappa}$ is
\begin{equation*}
    \breve{V}^d_t(\kappa)
    \coloneqq \frac{2\pi^{d/2}}{\kappa^{d-1}\Gamma(d/2)}
      \int_0^t \sinh^{d-1}(\kappa r)\,\mathrm{d}r,
\end{equation*}
independently of $x$; for $d = 2$ this reduces to $\breve{V}^2_t(\kappa) = 2\pi(\cosh(\kappa t) - 1)/\kappa^2$.

By a rescaling argument, the clustering tendency of $\hbb^d_{\kappa}$ at radius $t$ depends only on the product $\alpha \coloneqq \kappa t$, so we denote it by $\breve{p}_d(\alpha)$. For $d = 2$, one can show that
\begin{align}
\breve{p}_2(\alpha) = \int_0^1
\frac{2\alpha\sinh(\alpha x)}
{\pi(\cosh(\alpha)-1)^2}
\Bigg[
&\cosh(\alpha)
\arccos\left(\coth(\alpha)\tanh\left(\frac{\alpha x}{2}\right)\right)
\notag \\
&- \arctan\left(
\frac{
\sqrt{\cosh^2(\alpha) - \cosh^2(\frac{\alpha x}{2})}
}{
\sinh(\frac{\alpha x}{2})
}
\right)
\Bigg]\,\mathrm{d}x. \label{EqClustTendHyp}
\end{align}

The hyperbolic clustering tendency satisfies the following, proved in Appendix~\ref{SubsectionResultOnClustTend}.
\begin{proposition} \label{PropResultsOnClustTend}
    Let $d \geq 2$. The Euclidean and hyperbolic clustering tendencies
    satisfy:
    \begin{enumerate}[label=(\roman*)]
        \item $\breve{p}_d(\alpha) < \overline{p}_d$ for all $\alpha > 0$,
        \item $\alpha \mapsto \breve{p}_d(\alpha)$ is strictly decreasing on $(0,\infty)$,
        \item $\breve{p}_2(\alpha) = \overline{p}_2 - \bigl(\sqrt{3}/(16\pi)\bigr)\alpha^2 + O(\alpha^4)$ as $\alpha \to 0^+$,
        \item $\breve{p}_2(\alpha) = (8/\pi)e^{-\alpha/2} + o(e^{-\alpha/2})$ as $\alpha \to \infty$.
    \end{enumerate}
\end{proposition}

Consider the threshold model $\xi$ on $\hbb^d_{\kappa}$ with density $\delta > 0$ and threshold $t > 0$: points are placed according to a Poisson process with intensity measure $\delta\lambda_{\breve{g}}$, and two distinct points are adjacent exactly when their distance is at most $t$. As noted after Proposition~\ref{PropLocalIidRep}, $\xi$ admits no global i.i.d.\ representation, and by Proposition~\ref{PropThresholdModelNotExchangeable} no representation of $\xi$ has a jointly exchangeable adjacency matrix.

Instead, since its directing process is Poisson, $\xi$ is boundedly exchangeable (see the remark following Definition~\ref{DefBoundedExchangeability}). Again, every closed ball $B_R(x) \subseteq \hbb^d_{\kappa}$ with $R > 0$ identifies an infinite jointly exchangeable matrix $A^{B_R(x)}$ via Theorem~\ref{ThmUniqueLocalInfiniteExchangeableMatrix}, and the adjacency matrix of any i.i.d.\ representation of $\xi\vert_{B_R(x)}$ may be regarded as a submatrix of $A^{B_R(x)}$ of random finite dimension. Explicitly, since $\xi$ is connected by the threshold kernel,
\begin{equation*}
    A^{B_R(x)}_{ij} \overset{a.s.}{=} \mathbf{1}\{i \neq j\}\mathbf{1}\{d_{\breve{g}}(S_i,S_j) \leq t\}, \qquad i,j \in \nbb,
\end{equation*}
for an array $(S_i)$ of i.i.d.\ uniformly distributed random elements of $B_R(x)$. By homogeneity, the distribution of $A^{B_R(x)}$ does not depend on $x$, so it suffices to consider the hyperbolic origin $o_\kappa \coloneqq (\kappa^{-1}, 0, \dots, 0) $.

By Proposition~\ref{PropLocalGraphon}, $A^{B_R(o_\kappa)}$ admits an Aldous--Hoover--Kallenberg representation \eqref{EqAldousHooverKallenbergRepErgodic} with a local graphon $W_{B_R(o_\kappa)} : [0,1]^2 \to [0,1]$. Again, it suffices to choose a sampling function for $\xi$ on $B_R(o_\kappa)$ --- here, any measurable $f_{B_R(o_\kappa)} : [0,1] \to B_R(o_\kappa)$ such that $f_{B_R(o_\kappa)}(U)$ is uniformly distributed on $B_R(o_\kappa)$ when $U \sim \mathrm{U}[0,1]$ --- and to set
\begin{equation*}
    W_{B_R(o_\kappa)}(u,v) \coloneqq \mathbf{1}\{u \neq v\} \mathbf{1}\{d_{\breve{g}}(f_{B_R(o_\kappa)}(u),
    f_{B_R(o_\kappa)}(v)) \leq t\};
\end{equation*}
different choices of the sampling function yield local graphons that are equivalent in the sense of Proposition~\ref{PropLocalGraphonUniqueness}. Given the Riemannian volume measure in spherical polar coordinates, we construct a sampling function $f_{B_R(o_\kappa)}$ explicitly for $d = 2$. With $(U_1, U_2)$ obtained via the same binary expansion of $U \sim \mathrm{U}[0,1]$ as before,
\begin{equation} \label{EqHypfB}
    f_{B_R(o_\kappa)}(U) \coloneqq \frac{1}{\kappa} \begin{pmatrix}
    U_1(\cosh(\kappa R) - 1) + 1 \\
    \sqrt{(U_1(\cosh(\kappa R) - 1) + 1)^2 - 1}\, \cos(2\pi U_2) \\
    \sqrt{(U_1(\cosh(\kappa R) - 1) + 1)^2 - 1}\, \sin(2\pi U_2)
    \end{pmatrix}
\end{equation}
is uniformly distributed on $B_R(o_\kappa)$.

For $d = 2$, we again illustrate Proposition~\ref{PropConvClust} by simulation, now on $\hbb^2_2$. To compare against the Euclidean threshold model, we set $\delta = 1$ and $t = \arcosh(3)/2$, so that $\delta\breve{V}^2_t(2) = \pi$ and, by Corollary~\ref{CorEmpiricalDegreeDistribution}, the empirical degree distribution converges to the Poisson law with mean $\pi \approx 3.14$, as in the Euclidean case. The limiting clustering coefficient of Proposition~\ref{PropConvClust} is $\bigl(1 - e^{-\pi}(1+\pi)\bigr)\,\breve{p}_2(\arcosh(3)) \approx 0.404$, strictly below the Euclidean value $\approx 0.482$. Taking $E_n = B_n(0)$, condition~\eqref{CondSummable} holds since $\lambda_{\breve{g}}(B_n(o_\kappa))$ grows exponentially in $n$, so the convergence is almost sure. We sampled 1000 independent realizations of $G$ restricted to $B_{4.4}(o_\kappa)$ and computed the clustering coefficient $\widehat{C}_R$ of the vertices in $B_R(o_\kappa)$ for a range of radii $R \leq 3.5$; this ensures that the local clustering coefficients entering $\widehat{C}_R$ are free of boundary truncation. Figure~\ref{fig:hyp_clust} shows the concentration of $\widehat{C}_R$ around the limit, together with a single realization illustrating the almost-sure convergence.
\begin{figure}[h]
\centering
\includegraphics[width=.95\linewidth]{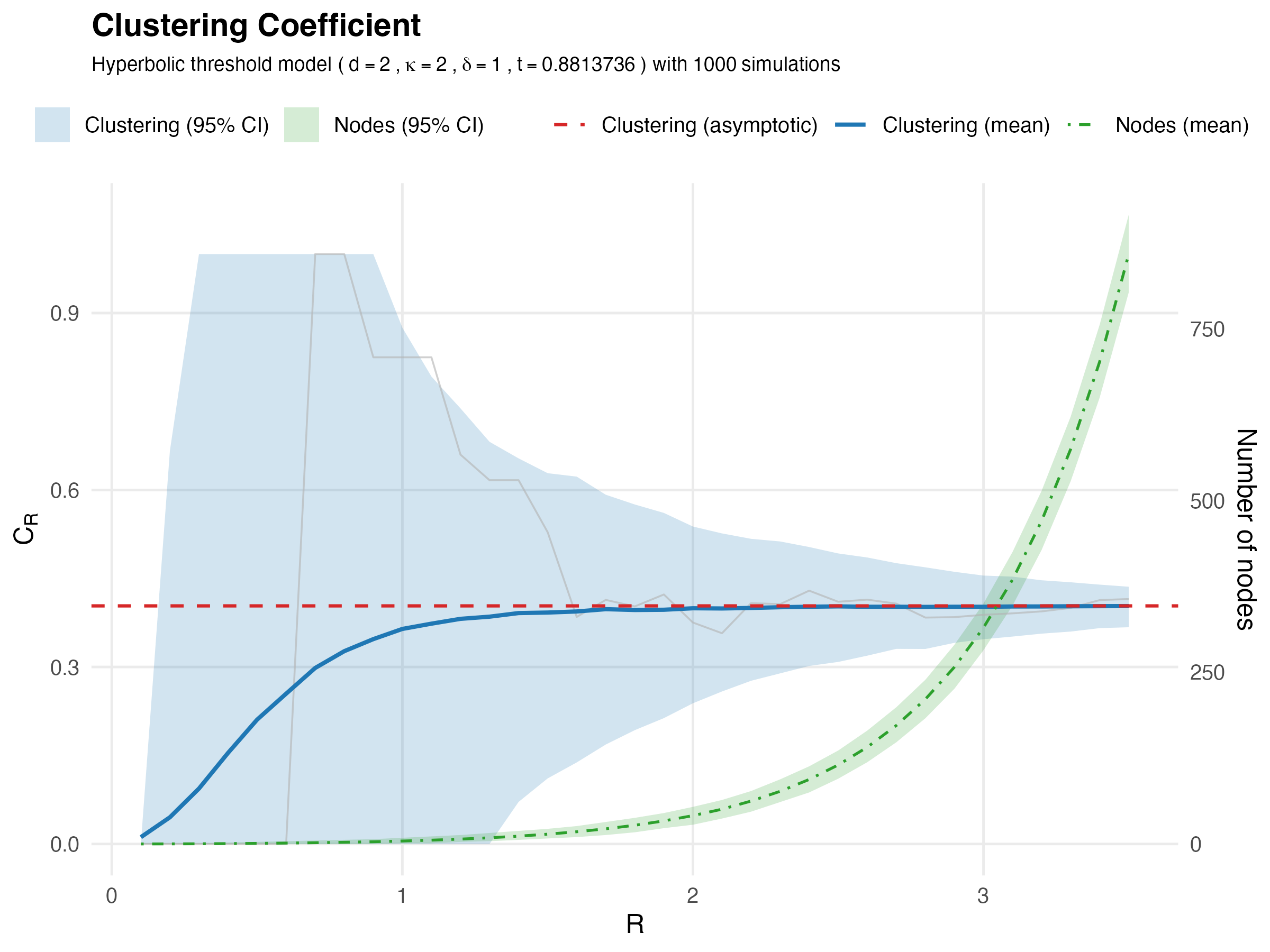}
\caption{Empirical clustering coefficient $\widehat{C}_R$ of the $2$-dimensional hyperbolic threshold model with $\delta = 1$ and $t = \arcosh(3)/2$. Dark blue: mean of $\widehat{C}_R$ over 1000 realizations; light-blue band: empirical $2.5\%$ and $97.5\%$ quantiles; grey: one randomly selected realization; dashed red: the asymptotic limit $\approx 0.404$ (Proposition~\ref{PropConvClust}). Dark green and light-green band (right axis): mean and corresponding quantiles of the number of vertices in $B_R(o_\kappa)$.}
\label{fig:hyp_clust}
\end{figure}

\begin{figure}[h]
\centering
\includegraphics[width=.6\linewidth]{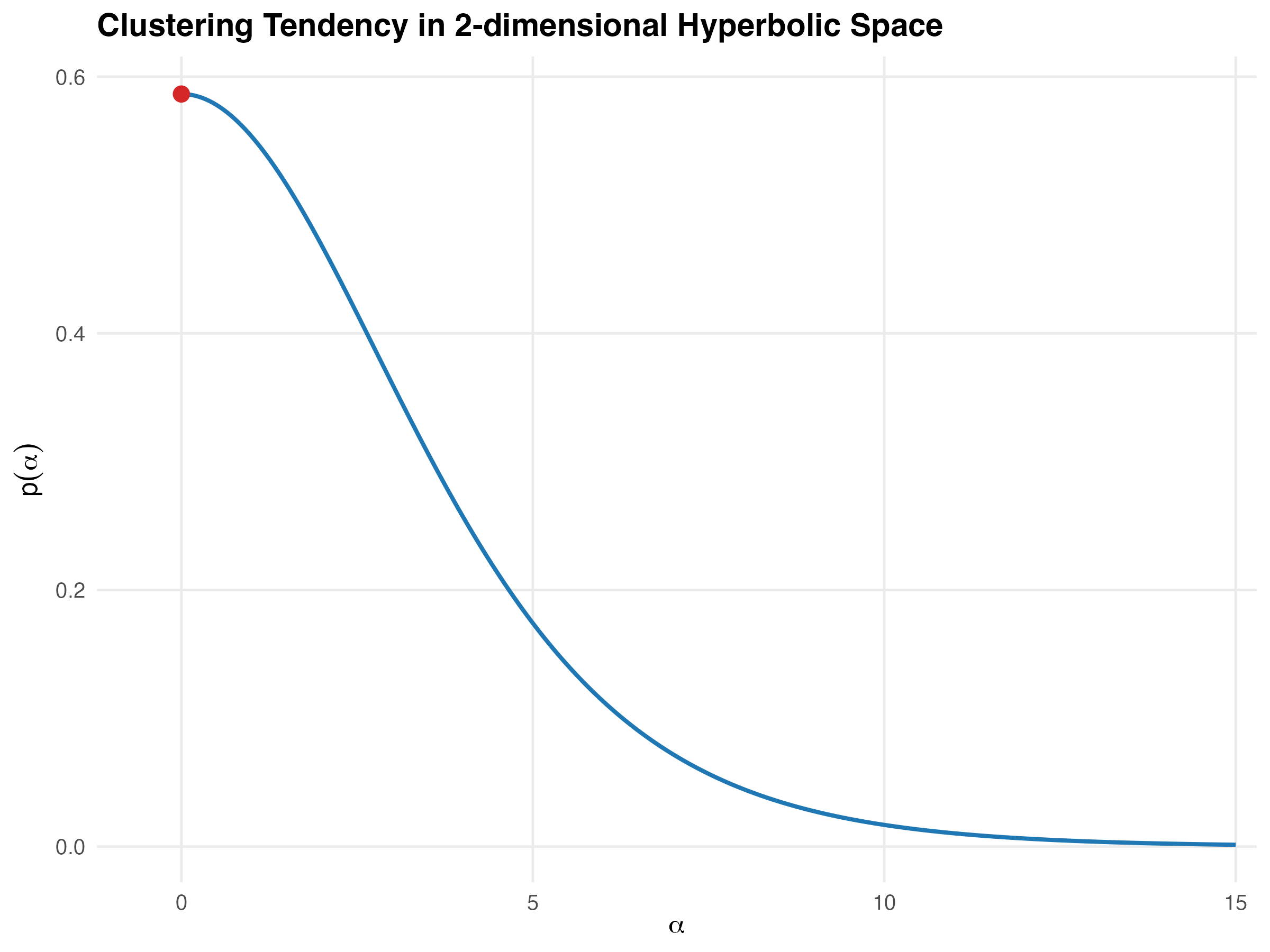}
\caption{The clustering tendency $\breve{p}_2(\alpha)$ of $2$-dimensional hyperbolic space, plotted as a function of $\alpha \coloneqq \kappa t$ (see~\eqref{EqClustTendHyp}). As $\alpha \to 0^+$ it converges to its Euclidean counterpart $\overline{p}_2 = 1 - 3\sqrt{3}/(4\pi) \approx 0.5865$, marked by the red dot, with quadratic approach by Proposition~\ref{PropResultsOnClustTend}(iii); as $\alpha \to \infty$ it decays to zero, exponentially by Proposition~\ref{PropResultsOnClustTend}(iv).}
\label{fig:p_clust}
\end{figure}

Although $G$ again has a Poisson degree distribution with mean $\pi$, it exhibits strictly lower asymptotic clustering than its Euclidean counterpart. This illustrates Proposition~\ref{PropResultsOnClustTend}(i): the hyperbolic clustering tendency is strictly below the Euclidean one, $\breve{p}_2(\alpha) < \overline{p}_2$ for all $\alpha > 0$ (see Figure~\ref{fig:p_clust}). Hyperbolic space therefore tunes the asymptotic clustering coefficient independently of the asymptotic degree distribution: at a matched degree distribution, the hyperbolic threshold model always produces a strictly lower asymptotic clustering coefficient.

\section{Conclusion}
\label{SectionConclusion}

This paper studies graph processes --- a class of random graph models in which vertices are identified with the points of a point process on a latent Polish space, and edges between distinct vertices are drawn conditionally independently via a connection kernel. Boundedly exchangeable graph processes recover classical matrix exchangeability on bounded observation windows: every nontrivial bounded restriction determines an infinite jointly exchangeable adjacency matrix, unique in distribution, and hence a local graphon representation. The global graph, by contrast, may be sparse and admit no jointly exchangeable representation. The threshold model on Riemannian manifolds serves as a geometrically natural example, demonstrating that the geometry of the latent space is directly reflected in the macroscopic properties of the network. Once the limiting degree distribution is matched across manifolds by a joint choice of density and threshold, curvature acts as an additional geometric parameter affecting the limiting clustering coefficient: at the matched degree law, hyperbolic space produces strictly smaller limiting clustering than Euclidean space. This makes graph processes a flexible and geometrically interpretable tool for modelling large sparse networks.

Several natural avenues for future research present themselves. On the modelling side, the asymptotic analysis of Section~\ref{SectionApplications} could be extended from the threshold model to general distance-based kernels $W(x,y) = h(d_g(x,y))$, whose connection probabilities decay smoothly with distance rather than cutting off sharply, or to manifolds of locally varying curvature. The latter would forgo the homogeneity underlying our basepoint-free degree and clustering statistics, but could let these structural properties --- such as clustering or degree heterogeneity --- vary across different parts of the graph. On the statistical side, the explicit dependence of the limiting degree and clustering statistics on curvature suggests using them to infer model parameters, and the geometry of the latent space itself, from observed network data --- a significant step towards practical applicability.

\section*{Acknowledgements}

The authors acknowledge the use of the AI-assisted tools ChatGPT (OpenAI; several versions of GPT-5; accessed via chatgpt.com) and Claude (Anthropic; Claude 4.8 Opus and Claude 5 Sonnet, accessed via claude.ai).

These tools were used to suggest improvements to wording and grammar, to flag possible minor errors or omissions in the mathematical exposition for subsequent independent review, and to assist with LaTeX table formatting and code used to format plots. The authors evaluated every suggestion and independently verified all mathematical statements, proofs, references, numerical values, tables, and figures. 

All scientific content, interpretations, and conclusions were fully created, verified, and approved by the authors. The authors take full responsibility for the content of the manuscript.

\section*{Funding information}

Leon-Roman Rinke's position at Leibniz Universität Hannover was fully financed through the dual PhD programme of the House of Insurance, which is supported by the insurance industry. This financing entailed no obligations regarding the research, and the supporting institutions had no role in the design or conduct of the research, its analysis or interpretation, or the preparation of this manuscript.

\section*{Competing interests}

There are no competing interests to declare which arose during the preparation or publication process of this article.

\printbibliography

\section{Appendix}
\label{SectionAppendix}

\subsection{Proof of Lemma~\ref{LemInvarianceMeasurePresTrans}}

\begin{proof}[Proof of Lemma~\ref{LemInvarianceMeasurePresTrans}]
For $m\in\nbb$, write $\Pcal_m$ for the $\sigma$-algebra generated by the dyadic partition of $[0,1]$ into $N_m \coloneqq 2^m$ intervals $\{I_{m,1},\ldots,I_{m,N_m}\}$ of equal length. The process
\[
    W_m
    \coloneqq
    \ebb\bigl[
        W\mid \Pcal_m^{\otimes 2}
    \bigr].
\]
is a martingale with respect to the filtration $(\Pcal_m^{\otimes 2})_{m \in \nbb}$. Observe that $W_m$ is constant on every rectangle $I_{m,i}\times I_{m,j}$. Since $W$ is bounded and the dyadic rectangles generate $\Bcal([0,1]^2)$, martingale convergence (see, e.g., \cite[Theorem~9.24]{Kallenberg2021}) yields
\[
    \lVert W_m-W\rVert_1\longrightarrow0.
\]

Every permutation $\sigma$ of $\{1,\ldots,N_m\}$ is induced, up to changes at finitely many endpoints, by a $\lambda$-preserving bijection $T_\sigma$ satisfying
\[
    T_\sigma(I_{m,i})=I_{m,\sigma(i)}.
\]
The assumed invariance of $W$ therefore implies that the averages of $W$ over $I_{m,i}\times I_{m,j}$ and $I_{m,\sigma(i)}\times I_{m,\sigma(j)}$ coincide. Since the symmetric group acts transitively on the diagonal pairs $(i,i)$ and on the off-diagonal pairs $(i,j)$, $i\neq j$, there exist $\alpha_m,\beta_m\in[0,1]$ such that
\[
    W_m
    =
    \alpha_m
    \quad\text{off }\Delta_m,
    \qquad
    W_m
    =
    \beta_m
    \quad\text{on }\Delta_m,
\]
where
\[
    \Delta_m
    \coloneqq
    \bigcup_{i=1}^{N_m}I_{m,i}^2,
    \qquad
    \lambda^{\otimes2}(\Delta_m)=\frac{1}{N_m}.
\]

Let $c\coloneqq\int_{[0,1]^2}W\,d\lambda^{\otimes2}$. Since conditional expectation preserves integrals,
\[
    c
    =
    \left(1-\frac{1}{N_m}\right)\alpha_m
    +
    \frac{1}{N_m}\beta_m.
\]
Hence $\vert\alpha_m-c\vert\leq N_m^{-1}$, while $\lVert W_m-\alpha_m\rVert_1\leq N_m^{-1}$. Consequently,
\[
    \lVert W_m-c\rVert_1
    \leq \frac{2}{N_m}
    \longrightarrow0.
\]
Together with $\lVert W_m-W\rVert_1\to0$, the triangle inequality gives $\lVert W-c\rVert_1=0$. Thus
\[
    W=c
    \qquad \lambda^{\otimes2}\text{-a.e.}
\]
\end{proof}

\subsection{Spherical polar coordinates in euclidean space} \label{SubsectionSphericalPolEuc}

In spherical polar coordinates, a point $(x_1, \dots, x_d) \in \rbb^d$ ($d \geq 2$) is given by
\begin{align*}
    x_1 &= r\cos(\theta_1) \\
    x_2 &= r\sin(\theta_1)\cos(\theta_2) \\
    x_3 &= r\sin(\theta_1)\sin(\theta_2)\cos(\theta_3) \\
    &\vdots \\
    x_{d-1} &= r\sin(\theta_1)\cdots \sin(\theta_{d-2})\cos(\theta_{d-1}) \\
    x_d &= r\sin(\theta_1)\cdots \sin(\theta_{d-2})\sin(\theta_{d-1})
\end{align*}
for $r \in (0, \infty)$, $\theta_1, \dots, \theta_{d-2} \in (0, \pi)$, and $\theta_{d-1} \in (0, 2\pi)$, with the convention that for $d = 2$ the only angle is $\theta_1 \in (0, 2\pi)$. This parametrizes $\rbb^d$ up to a $\lambda_d$-null set, namely the half-hyperplane $\{x_d = 0,\ x_{d-1} \geq 0\}$.

In these coordinates, the volume element of the Lebesgue measure $\lambda_d$ on $\rbb^d$ satisfies
\begin{equation*}
    \mathrm{d}\lambda_d = r^{d-1}\sin^{d-2}(\theta_1)\sin^{d-3}(\theta_2)
    \cdots\sin(\theta_{d-2})\,\mathrm{d}r\,\mathrm{d}\theta_1\cdots
    \mathrm{d}\theta_{d-1}.
\end{equation*}
We write
\begin{equation*}
    \mathrm{d}\omega \coloneqq \sin^{d-2}(\theta_1)\sin^{d-3}(\theta_2)
    \cdots\sin(\theta_{d-2})\,\mathrm{d}\theta_1\cdots\mathrm{d}\theta_{d-1},
\end{equation*}
the surface area element of the unit sphere $S^{d-1}$, with total mass $\int_{S^{d-1}}\mathrm{d}\omega = 2\pi^{d/2}/\Gamma(d/2)$, to obtain
\begin{equation*}
    \mathrm{d}\lambda_d = r^{d-1}\,\mathrm{d}r\,\mathrm{d}\omega.
\end{equation*}

\subsection{Spherical polar coordinates in hyperbolic space} \label{SubsectionSphericalPolHyp}

We proceed similarly to Euclidean space. Define a point $(x_0, x_1, \dots, x_d) \in \hbb^d_{\kappa}$ ($d \geq 2$) by
\begin{align*}
    x_0 &= \frac{1}{\kappa}\cosh(\kappa r) \\
    x_1 &= \frac{1}{\kappa}\sinh(\kappa r)\cos(\theta_1) \\
    x_2 &= \frac{1}{\kappa}\sinh(\kappa r)\sin(\theta_1)\cos(\theta_2) \\
    &\vdots \\
    x_{d-1} &= \frac{1}{\kappa}\sinh(\kappa r)\sin(\theta_1)\cdots
    \sin(\theta_{d-2})\cos(\theta_{d-1}) \\
    x_d &= \frac{1}{\kappa}\sinh(\kappa r)\sin(\theta_1)\cdots
    \sin(\theta_{d-2})\sin(\theta_{d-1})
\end{align*}
for $r \in (0, \infty)$, $\theta_1, \dots, \theta_{d-2} \in (0, \pi)$, and $\theta_{d-1} \in (0, 2\pi)$, with the convention that for $d = 2$ the only angle is $\theta_1 \in (0, 2\pi)$. This parametrizes $\hbb^d_{\kappa}$ up to a $\lambda_{\breve{g}}$-null set, namely $\{x \in \hbb^d_{\kappa} : x_d = 0,\ x_{d-1} \geq 0\}$.

In these coordinates, the volume element of the Riemannian volume measure $\lambda_{\breve{g}}$ on $\hbb^d_{\kappa}$ satisfies
\begin{equation*}
    \mathrm{d}\lambda_{\breve{g}} = \frac{1}{\kappa^{d-1}}
    \sinh^{d-1}(\kappa r)\sin^{d-2}(\theta_1)\sin^{d-3}(\theta_2)
    \cdots\sin(\theta_{d-2})\,\mathrm{d}r\,\mathrm{d}\theta_1\cdots
    \mathrm{d}\theta_{d-1}.
\end{equation*}
With $\mathrm{d}\omega$ the surface area element of the unit sphere $S^{d-1}$ from Appendix~\ref{SubsectionSphericalPolEuc} this becomes
\begin{equation*}
    \mathrm{d}\lambda_{\breve{g}} = \frac{1}{\kappa^{d-1}}
    \sinh^{d-1}(\kappa r)\,\mathrm{d}r\,\mathrm{d}\omega.
\end{equation*}

\subsection{Proof of Proposition~\ref{PropResultsOnClustTend}} \label{SubsectionResultOnClustTend}

\begin{proof}[Proof of Proposition~\ref{PropResultsOnClustTend}]
We start with the first part. By scale invariance, we may normalize the threshold to $1$. Use spherical polar coordinates around the centre of the ball, and write the radial coordinates of two sampled points as $r,s\in[0,1]$ and the smaller angle between their radial directions as $\theta\in[0,\pi]$. In the Euclidean ball, the radial distribution function is $F_0(r)=r^d$. In the hyperbolic
ball, it is
\[
        F_\alpha(r)
        =
        \frac{\int_0^r \sinh^{d-1}(\alpha u)\,du}
        {\int_0^1 \sinh^{d-1}(\alpha u)\,du}.
\]
Since $u\mapsto \sinh(\alpha u)/u$ is strictly increasing on $(0,\infty)$,
\[
        F_\alpha(r)
        =
        \frac{\int_0^r u^{d-1}\left(\frac{\sinh(\alpha u)}{u}\right)^{d-1}\,du}
        {\int_0^1 u^{d-1}\left(\frac{\sinh(\alpha u)}{u}\right)^{d-1}\,du}
        \le
        \frac{\int_0^r u^{d-1}\,du}{\int_0^1 u^{d-1}\,du}
        =
        r^d
        =
        F_0(r),
\]
with strict inequality for $r \in (0,1)$ and $\alpha>0$. 

Let $q_0(r,s)$ and $q_\alpha(r,s)$ be the conditional probabilities, over the angle, that two points at radii $r,s$ are at distance at most $1$ in the Euclidean and hyperbolic balls, respectively. We claim that $q_\alpha(r,s) \leq q_0(r,s)$. When $r+s \leq 1$, this follows trivially from the triangle inequality. Otherwise, let $\gamma_0$ and $\gamma_\alpha$ be the Euclidean and hyperbolic angles opposite a side of length $1$ in triangles with remaining side lengths $r,s$. Then $d_0\le 1$ is equivalent to $\theta\le\gamma_0$, and $d_\alpha\le 1$ is equivalent to $\theta\le\gamma_\alpha$. It remains to compare the two critical angles. Set $\sigma=(r+s+1)/2$. The Euclidean and hyperbolic laws of cosines give
\[
        \sin^2\frac{\gamma_0}{2}
        =
        \frac{(\sigma-r)(\sigma-s)}{rs},
        \qquad
        \sin^2\frac{\gamma_\alpha}{2}
        =
        \frac{
        \sinh(\alpha(\sigma-r))\sinh(\alpha(\sigma-s))
        }{
        \sinh(\alpha r)\sinh(\alpha s)
        }.
\]
Since $r,s\le 1$ and $r+s>1$, we have $0<\sigma-r\le s$ and $0<\sigma-s\le r$. Again using that $y\mapsto \sinh(\alpha y)/y$ is strictly increasing, we obtain
\[
        \frac{\sinh(\alpha(\sigma-r))}{\sigma-r}
        \le
        \frac{\sinh(\alpha s)}{s},
        \qquad
        \frac{\sinh(\alpha(\sigma-s))}{\sigma-s}
        \le
        \frac{\sinh(\alpha r)}{r}.
\]
Multiplying these inequalities yields
\[
        \sin^2\frac{\gamma_\alpha}{2}
        \le
        \sin^2\frac{\gamma_0}{2}.
\]
Both angles lie in $[0,\pi]$, so $\gamma_\alpha\le\gamma_0$. Since the distribution of $\theta$ is the same in both cases, this proves $q_\alpha(r,s)\le q_0(r,s)$.

Couple the radial variables by independent $U,V\sim \mathrm{U}[0,1]$, setting
\[
        R_0=F_0^{-1}(U),\quad S_0=F_0^{-1}(V),
        \qquad
        R_\alpha=F_\alpha^{-1}(U),\quad S_\alpha=F_\alpha^{-1}(V).
\]
Since $F_\alpha<F_0$ on $(0,1)$, we have $R_\alpha>R_0$ and $S_\alpha>S_0$ almost surely. Therefore
\[
        \breve p_d(\alpha)
        =
        \ebb[q_\alpha(R_\alpha,S_\alpha)]
        \le
        \ebb[q_0(R_\alpha,S_\alpha)]
        \le
        \ebb[q_0(R_0,S_0)]
        =
        \overline{p}_d.
\]
The second inequality is strict because $\pbb(R_0+S_0>1)>0$, and on this event $q_0$ is strictly decreasing in both arguments while $R_\alpha>R_0$ and $S_\alpha>S_0$ almost surely. Hence $\breve p_d(\alpha)<\overline{p}_d$. Using the same line of reasoning one shows that $\breve{p}_d(\alpha') < \breve{p}_d(\alpha)$ when $\alpha' > \alpha$.

For the third claim, write
\[
\breve{p}_2(\alpha)
=
\int_0^1 P_\alpha(x)
\bigl(A_\alpha(x)-B_\alpha(x)\bigr)\,dx,
\]
where
\[
P_\alpha(x)
\coloneqq
\frac{2\alpha\sinh(\alpha x)}
{\pi(\cosh\alpha-1)^2},
\]
\[
A_\alpha(x)
\coloneqq
\cosh(\alpha)
\arccos\!\left(
\coth(\alpha)\tanh\!\left(\frac{\alpha x}{2}\right)
\right),
\]
and
\[
B_\alpha(x)
\coloneqq
\arctan\!\left(
\frac{
\sqrt{\cosh^2(\alpha)-\cosh^2(\alpha x/2)}
}{
\sinh(\alpha x/2)
}
\right).
\]
We expand each term as $\alpha \to 0^+$. For $P_{\alpha}(x)$ we obtain
\[
P_\alpha(x)
=
\frac{8x}{\pi\alpha^2}
+
\frac{4x(x^2-1)}{3\pi}
+
O(\alpha^2).
\]
For $A_{\alpha}(x)$ we obtain
\[
A_\alpha(x)
=
\arccos\!\left(\frac{x}{2}\right)
+
\alpha^2 A_2(x)
+
\alpha^4 A_4(x)
+
O(\alpha^6),
\]
where
\[
A_2(x)
=
\frac12\arccos\!\left(\frac{x}{2}\right)
-
\frac{x\sqrt{4-x^2}}{12},
\]
and
\[
A_4(x)
=
\frac{
7x^3\sqrt{4-x^2}
-
52x\sqrt{4-x^2}
+
60\arccos(x/2)
}{1440}.
\]
For $B_{\alpha}(x)$ we obtain
\[
B_\alpha(x)
=
\arccos\!\left(\frac{x}{2}\right)
+
\alpha^2 B_2(x)
+
\alpha^4 B_4(x)
+
O(\alpha^6),
\]
where
\[
B_2(x)
=
\frac{x\sqrt{4-x^2}}{24},
\]
and
\[
B_4(x)
=
-
\frac{x(x^2+14)\sqrt{4-x^2}}{2880}.
\]
It follows that
\[
P_\alpha(x)\bigl(A_\alpha(x)-B_\alpha(x)\bigr)
=
F_0(x)+\alpha^2F_2(x)+O(\alpha^4),
\]
uniformly on $[0,1]$, with
\[
F_0(x)
=
\frac{1}{\pi}\left( 4x\arccos\!\left(\frac{x}{2}\right)
-
x^2\sqrt{4-x^2}\right),
\]
and
\[
F_2(x)
=
\frac{
16x^3\arccos(x/2)
-
8x\arccos(x/2)
-
3x^4\sqrt{4-x^2}
-
2x^2\sqrt{4-x^2}
}{24\pi}.
\]
We now integrate these terms. Observe that
\[
\int_0^1 F_0(x)\,dx
= \overline{p}_2.
\]
For the correction term, direct integration gives
\[
\int_0^1 F_2(x)\,dx
=
-\frac{\sqrt3}{16\pi},
\]
so
\[
\breve{p}_2(\alpha)
=
\overline{p}_2
-
\frac{\sqrt3}{16\pi}\alpha^2
+
O(\alpha^4)
\qquad \text{as } \alpha\to 0^+.
\]

For the final part, note that $0\le B_\alpha(x)\le \pi/2$, so
\[
0
\le
\int_0^1
\frac{2\alpha\sinh(\alpha x)}
{\pi(\cosh(\alpha)-1)^2}
B_\alpha(x)\,dx
\le
\frac{1}{\cosh(\alpha)-1}
=
O(e^{-\alpha}).
\]
It remains to analyse the $A_\alpha$-term. Setting $y = \alpha (1-x)$ gives
\[
I_\alpha
\coloneqq
\int_0^1
\frac{2\alpha\sinh(\alpha x)}
{\pi(\cosh(\alpha)-1)^2}
A_\alpha(x)\,dx
=
\int_0^\alpha
\frac{2\sinh(\alpha-y)}
{\pi(\cosh(\alpha)-1)^2}
A_\alpha\!\left(1-\frac{y}{\alpha}\right)\,dy.
\]
For fixed $y \geq 0$, as $\alpha\to\infty$,
\[
\frac{2\sinh(\alpha-y)\cosh(\alpha)}
{\pi(\cosh(\alpha)-1)^2}
=
\frac{2}{\pi}e^{-y} + O(e^{-\alpha}).
\]
Moreover,
\[
\coth(\alpha)
\tanh\!\left(\frac{\alpha-y}{2}\right)
=
1-2e^{y-\alpha}+O(e^{-2\alpha}),
\]
and therefore
\[
\arccos\!\left(
\coth(\alpha)
\tanh\!\left(\frac{\alpha-y}{2}\right)
\right)
=
2e^{(y-\alpha)/2} + O(e^{-\alpha}).
\]
Consequently,
\[
e^{\alpha/2}
\frac{2\sinh(\alpha-y)}
{\pi(\cosh(\alpha)-1)^2}
A_\alpha\!\left(1-\frac{y}{\alpha}\right)
\to
\frac{4}{\pi}e^{-y/2}.
\]
pointwise on $y \geq 0$. By dominated convergence,
\[
\lim_{\alpha\to\infty} e^{\alpha/2} I_\alpha
=
\int_0^\infty \frac{4}{\pi}e^{-y/2}\,dy
=
\frac{8}{\pi}.
\]
Combining this with the estimate for the $B_\alpha$-term gives
\[
\breve{p}_2(\alpha)
=
\frac{8}{\pi}e^{-\alpha/2}
+
o(e^{-\alpha/2})
\qquad \text{as } \alpha\to\infty.
\]
\end{proof}

\end{document}